\documentclass[11pt]{article}

\usepackage{dsfont}

\usepackage{psfrag}

\usepackage{calrsfs}
\usepackage{mathrsfs}
\usepackage{pdfsync}
\usepackage{amsmath, amsthm, amssymb, amsfonts} 
\usepackage[shortlabels]{enumitem}

\usepackage{url}
\usepackage{float}

\usepackage[usenames]{color}

\usepackage{tikz,fullpage}

\usetikzlibrary{arrows,decorations.pathmorphing,backgrounds,positioning,fit,automata}
\usepackage{pgfplots}
\pgfplotsset{compat=1.15}

\usepackage{indentfirst,calc,euscript}
\usepackage{setspace}
\usepackage[reals]{layout}
\usepackage{xr}
\usepackage{amscd}

\usepackage{sgame}
\usepackage{subfigure}

\usepackage[colorlinks=true,breaklinks=true,bookmarks=true,urlcolor=blue,
     citecolor=blue,linkcolor=blue,bookmarksopen=false,draft=false]{hyperref}

\newcommand{\ignore}[1]{}

\newtheorem{theorem}{Theorem}

\newtheorem{lemma}[theorem]{Lemma}

\newtheorem{proposition}[theorem]{Proposition}

\theoremstyle{definition}

\newtheorem{definition}[theorem]{Definition}

\newtheorem{remark}[theorem]{Remark}

\newtheorem{assumption}{Assumption}

\numberwithin{equation}{section}

\numberwithin{theorem}{section}

\newcommand{\m}{\mathbb}

\DeclareMathOperator*{\supp}{supp}
\DeclareMathOperator*{\val}{val}
\DeclareMathOperator*{\maxmin}{maxmin}

\newcommand\numberthis{\addtocounter{equation}{1}\tag{\theequation}}

\author{Bruno Ziliotto}
\title{Mertens conjectures in absorbing games with incomplete information}
\date{}
\begin{document}
\maketitle
\bibliographystyle{plain}
\begin{abstract}
In a zero-sum stochastic game with signals \cite[Chapter IV]{MSZ}, at each stage, two adversary players take decisions and receive a stage payoff determined by these decisions and a variable called \textit{state}. The state follows a Markov chain, that is controlled by both players. Actions and states are imperfectly observed by players, who receive a private signal at each stage. Mertens \cite{M86} conjectured two properties regarding games with long duration: first, that limit value always exists, second, that when Player 1 is more informed than Player 2, she can guarantee uniformly the limit value. These conjectures were disproved recently by the author \cite{Z13}, but remain widely open in many subclasses. A well-known particular subclass is the one of \textit{absorbing games with incomplete information on both sides}, in which the state can move at most once during the game, and players get a private signal about it at the outset of the game. This paper proves Mertens conjectures in this particular model, by introducing a new approximation technique of belief dynamics, that is likely to generalize to many other frameworks. In particular, this makes a significant step towards the understanding of the following broad question: in which games do Mertens conjectures hold? 
\end{abstract}
\section*{Introduction}
Discrete-time stochastic games describe repeated interactions between players in a changing environment, and are a widely studied subject in Game Theory \cite{A03,YZ14,JN16,JN16b,LS20}. The first model of this kind (\textit{standard stochastic games}) was introduced by Shapley \cite{SH53}. It features two players that take simultaneous decisions and receive opposite payoffs $g_m$ and $-g_m$ at each stage $m \geq 1$. Payoff $g_m$ is determined by these decisions and a variable called \textit{state of nature}. The state of nature is known and follows a Markov chain controlled by both players. State and action sets are finite. Zero-sum stochastic games can be viewed as a generalization of Markov chains (``0-Player case") and Markov Decision Processes \cite{Bellman_57} (``1-Player case"). 

In the \textit{$n$-stage game}, $n \geq 1$, the total payoff is the expected Ces\`aro mean of the stage payoffs $\frac{1}{n} \sum_{m=1}^n g_m$, and in the \textit{$\lambda$-discounted game}, the total payoff is the expected $\lambda$-Abel mean of the stage payoffs $\sum_{m \geq 1} \lambda(1-\lambda)^{m-1}g_m$. Maxmin and minmax coincide in the $n$-stage game, and they are called the \textit{value} of the game, denoted by $v_n$. Similarly, the value of the $\lambda$-discounted game is denoted by $v_\lambda$. Intuitively, the value corresponds to Player 1's payoff outcome when both players play rationally. Investigating properties of \textit{long} stochastic games have been a primary focus of literature. Formally, this corresponds to the asymptotic regime $n \rightarrow +\infty$ and $\lambda \rightarrow 0$. Bewley and Kohlberg \cite{BK76} have proved that $(v_n)$ and $(v_\lambda)$ converge to the same limit, called \textit{limit value}. Another related model where such a result has been proven true is the \textit{repeated game with incomplete information on both sides model} \cite{AM95, MZ71}, in which the state of nature never moves, and players get a single private signal about it at the outset of the game. 

When the limit value exists, a natural question is to ask for ``good" strategies that are robust with respect to the duration of the game (\textit{uniform approach}): does Player 1 (or Player 2) has a strategy that is approximately optimal in any $n$-stage game and $\lambda$-discounted game, provided that $n$ is large enough and $\lambda$ is small enough? Such a property holds in standard stochastic games \cite{MN81} and in repeated games with incomplete information on one side \cite{AM95}. These seminal results have inspired a huge amount of work on limit value and uniform approach in related models (see \cite{SZ16,JN16b,LS20} for recent surveys on the topic). A general model of stochastic game with signals was formulated by Mertens, Sorin and Zamir \cite[Chapter IV]{MSZ}, that includes most of the models studied in literature: in such a model, players may not know states and past actions, and receive private signals at every stage. Two influential conjectures were formulated by Mertens \cite{M86}: \textit{in any stochastic game with signals, the limit value exists, and moreover, when Player 1 is more informed than Player 2, she can guarantee uniformly the limit value.} These conjectures motivated a considerable literature and were proven true in numerous cases (see \cite{SZ16} for a recent survey). Nonetheless, the author disproved them recently \cite{Z13}, by providing a stochastic game with public signals on state without a limit value. Since then, positive results have been found in particular cases \cite{LV14,Li20,garrec19,LR20}, but the general question of characterizing stochastic games that satisfy Mertens conjectures remains largely uncharted. 

One particular case that has captured a lot of attention is the absorbing games with incomplete information on both sides class. An absorbing game is a stochastic game in which all states but one are \textit{absorbing}: once in an absorbing state, it stays there, irrespective of players' actions. An absorbing game with incomplete information on both sides is an absorbing game that depends on a fixed parameter that is unknown to players, and such that at the outset of the game, players get one single private signal on this parameter. In the complete information case (both players know the parameter), Mertens conjectures hold thanks to Kohlberg \cite{kohlberg1974}, and in the one-sided case (Player 1 knows the parameter), they hold in ``generalized Big Match" \cite{S841,S852,Li20}, that extend on the classic Big Match game \cite{BF68}. In the general one-sided case, Rosenberg \cite{rosenberg00} proved existence of a limit value. In an unpublished working paper, Laraki wrote an incomplete proof of existence of the limit value in the general two-sided case. To the best of our knowledge, this proof has not been corrected, and is independent of this paper.  

This paper proves Mertens conjectures in the general model of absorbing games with incomplete information on both sides. In addition to solving a well-known model in literature, it provides a new methodological approach for adressing Mertens conjectures in stochastic games. 
Indeed, the proof relies on the approximation of the original game by an auxiliary stochastic game with finite state space and compact action sets, in which states correspond to beliefs and are observed. Such a stochastic game has \textit{semi-algebraic separable transitions}, hence satisfies the Mertens conjectures \cite{BGV13}. The main difficulty is that this approximation has to be uniform in time, meaning that state dynamics in the original game and in the auxiliary game have to be close to each other in some sense, \textit{at any stage}. Even in the 0-Player case, such a property would already be very demanding, since the law of the infinite state sequence generated by a Markov chain is usually not robust to perturbations of the kernel. Hence, this approximation has to be done in a precise and game-adapted way, and consists in splitting at each stage the belief of players in the original game into vertices of a triangulation of the belief simplices. A coupling argument then shows that, in rough terms, the error between the original game and the auxiliary game
is bounded by a quantity that is small with respect to the $L^1$-variation of belief martingales of players. Then, the problem reduces to finding approximately optimal strategies that generate belief martingales with bounded $L^1$-variation. This point turns out to be quite delicate, and requires the introduction of a new class of strategies, called \textit{concise strategies}.

 Interestingly, most tools and lemmas introduced in the proof do not depend on the fact that the game is absorbing. That is why this new methodology is likely to be extended to many other frameworks. Building on this fact, the author proposes a new broad class of stochastic games with signals that is conjectured to satisfy Mertens statements.

The paper is organized as follows. Section \ref{sec:model} presents the model, the main results and its relation to Mertens conjectures, and provides a sketch of proof. Section \ref{sec:auxiliary} presents the auxiliary stochastic game. Section \ref{sec:concise} proves existence of approximately optimal strategies in the discounted game, that generate belief martingales with bounded variation. Section \ref{sec:coupling} formalizes the coupling argument and proves existence of the limit value. Section \ref{sec:uniform} proves the second Mertens conjecture. Section \ref{sec:perspectives} discusses possible extensions. 
\subsection*{Notations}
Throughout the paper, $\m{N}$ designates the set of nonnegative integers, and $\m{N}^*:=\m{N} \setminus \left\{0\right\}$. The set of real numbers is denoted by $\m{R}$. For $a,b \in \m{N}$, the notation $[a\dots b]$ stands for the set of integers larger or equal to $a$ and smaller or equal to $b$. 

When $C$ is a compact subset of a finite dimensional space, it will always be equipped with its Borelian $\sigma$-algebra, and the notation $\Delta(C)$ designates the set of probability measures over $C$. The set $\Delta(C)$ will be equipped with the Kantorovich-Rubinstein distance, which makes it a compact set. For $c\in C$, $\delta_c \in \Delta(C)$ designates the Dirac measure at $c$. The support of a probability measure $\mu \in \Delta(C)$ is denoted by $\supp(\mu)$. 

\section{Model and results} \label{sec:model}
\subsection{Stochastic games with incomplete information on both sides} \label{subsec:model}
A stochastic game with incomplete information on both sides is a tuple $\Gamma=(K,L,\Omega,I,J,\rho,g)$, where $K$ is Player 1's type space, 
$L$ is Player 2's type space, $\Omega$ is the state space, $I$ is Player 1's action set, $J$ is Player 2's action set, $\rho : \Omega \times I \times J \rightarrow \Delta(\Omega)$ is the transition function, $g: K \times L \times \Omega \times I \times J \rightarrow \m{R}$ is the payoff function. The sets $K,L,\Omega,I$ and $J$ are assumed to be finite. 

A state $\omega \in \Omega$ is \textit{absorbing} if for all $(k,\ell) \in K \times L$, $(i,j) \rightarrow g(k,\ell,\omega,i,j)$ is constant, and for all $(i,j) \in I \times J$, $\rho(\omega|\omega,i,j)=1$, and is \textit{non-absorbing} otherwise. 
When only one state is non-absorbing, the game is called \textit{absorbing game with incomplete information on both sides}, and this state is denoted by $\omega^0$. In the even more particular case where $L$ is a singleton, the game is called \textit{absorbing game with incomplete information on one side}. 
\\

Let $\Gamma=(K,L,\Omega,I,J,\rho,g)$ be a stochastic game with incomplete information on both sides. 
Given a pair of priors $(p,q) \in \Delta(K) \times \Delta(L)$ on types, and an initial state $\omega_1$, the game proceeds as follows: 
\begin{itemize}
\item
Player 1's type $k$ is drawn according to $p$, and Player 2's type $\ell$ is drawn independently according to $q$. Player 1 is informed of her type $k$, and Player 2 is informed of his type $\ell$.
\item
At each stage $m \geq 1$, simultaneously, Player 1 chooses $i_m \in I $ and Player 2 chooses $j_m \in J$. The stage payoff is $g(k,\ell,\omega_m,i_m,j_m)$, meaning that Player 1 receives $g(k,\ell,\omega_m,i_m,j_m)$, and Player 2 receives $-g(k,\ell,\omega_m,i_m,j_m)$. 
\item
A new state $\omega_{m+1}$ is drawn according to $\rho(\omega_{m},i_m,j_m)$, and $(i_m,j_m,\omega_{m+1})$ is announced to the players.
\end{itemize} 
A behavior strategy (resp., pure strategy) for Player 1 is a mapping 
\\
$\sigma: \cup_{m \geq 1} K \times (\Omega \times I \times J)^{m-1} \times \Omega \rightarrow \Delta(I)$ 
(resp., $\sigma: \cup_{m \geq 1} K \times (\Omega \times I \times J)^{m-1} \times \Omega \rightarrow I$). 
The interpretation of a strategy $\sigma$ is that at the beginning of stage $m \geq 1$, Player 1 knows her type $k$, as well as the sequence of past states and actions $(\omega_1,i_1,j_1,\dots,\omega_{m-1},i_{m-1},j_{m-1})$ and the current state $\omega_m$. Then, Player 1 draws an action according to the distribution $\sigma(k,\omega_1,i_1,j_1,\dots,\omega_{m-1},i_{m-1},j_{m-1},\omega_m)$. 

A behavior strategy (resp., pure strategy) for Player 2 is a mapping $\tau: \cup_{m \geq 1} L \times (\Omega \times I \times J)^{m-1} \times \Omega \rightarrow \Delta(J)$ 
(resp., $\tau: \cup_{m \geq 1} L \times (\Omega \times I \times J)^{m-1} \times \Omega \rightarrow J$). The set of behavior strategies of Player 1 (resp., Player 2) is denoted by $\Sigma$ (resp., $T$). 
A tuple $(p,q,\omega,\sigma,\tau) \in \Delta(K) \times \Delta(L) \times \Omega \times \Sigma \times T$ induces a probability measure $\m{P}_{p,q,\omega,\sigma,\tau}$ on the set of \textit{infinite histories} of the game
$K \times L \times (\Omega \times I \times J)^{\m{N}}$, and the expectation with respect to this probability measure is denoted by $\m{E}_{p,q,\omega,\sigma,\tau}$. 
Given $\lambda \in (0,1]$, the $\lambda$-discounted game $\Gamma_{\lambda}(p,q,\omega)$ is the game with payoff 
\begin{equation*}
\gamma_\lambda(p,q,\omega,\sigma,\tau):=\m{E}_{p,q,\omega,\sigma,\tau} \left( \sum_{m \geq 1} \lambda(1-\lambda)^{m-1} g(k,\ell,\omega_m,i_m,j_m) \right).
\end{equation*}
The value of this game exists \cite[Section IV.1.c, p. 174]{MSZ}, and is denoted by $v_{\lambda}(p,q,\omega)$:
\begin{equation*}
v_{\lambda}(p,q,\omega):=\max_{\sigma \in \Sigma} \min_{\tau \in T} \gamma_\lambda(p,q,\omega,\sigma,\tau)=\min_{\tau \in T} \max_{\sigma \in \Sigma} \gamma_\lambda(p,q,\omega,\sigma,\tau).
\end{equation*}
Denote $\left\|.\right\|_1$ the 1-norm, $\left\|.\right\|_2$ the 2-norm, and $\left\|.\right\|_\infty$ the uniform norm. 
When $\Delta(K)$ and $\Delta(L)$ are equipped with $\left\|.\right\|_1$, $v_\lambda$ is $\left\|g\right\|_{\infty}$-Lipschitz with respect to variables $p$ and $q$, concave with respect to $p$, and convex with respect to $q$.  

Given $n \in \m{N}^*$, the $n$-stage game $\Gamma_n(p,q,\omega)$ is the game with payoff 
\begin{equation*}
\gamma_n(p,q,\omega,\sigma,\tau):=\m{E}_{p,q,\omega,\sigma,\tau} \left(\frac{1}{n} \sum_{m = 1}^n g(k,\ell,\omega_m,i_m,j_m) \right).
\end{equation*}
The value of this game exists \cite[Section IV.1.c, p. 174]{MSZ}, and is denoted by $v_n(p,q,\omega)$:
\begin{equation*}
v_{n}(p,q,\omega):=\max_{\sigma \in \Sigma} \min_{\tau \in T} \gamma_n(p,q,\omega,\sigma,\tau)=\min_{\tau \in T} \max_{\sigma \in \Sigma} \gamma_n(p,q,\omega,\sigma,\tau).
\end{equation*}
\subsection{Results}
The two results of this paper are the following:
\begin{theorem} \label{main_thm}
In any absorbing game with incomplete information on both sides, $(v_\lambda)$ and $(v_n)$ converge uniformly to the same limit. 
\end{theorem}
\begin{remark} \label{rem:tauberian}
The author \cite{Z15} proved that in the general framework of \textit{stochastic games with signals}, which includes the model studied in this paper, $(v_\lambda)$ converges uniformly if and only if $(v_n)$ converges uniformly, and that both limits coincide. In this paper, we prove that $(v_\lambda)$ converges uniformly, which is thus enough to prove the above theorem. Recall that the common limit is called \textit{limit value}. 
\end{remark}
\begin{theorem} \label{main_thm_2}
Consider an absorbing game with incomplete information on one side, with limit value $v^*$. For any $(p,\omega) \in \Delta(K) \times \Omega$, for any $\varepsilon>0$, there exists $\sigma \in \Sigma$, $n_0 \geq 1$ and $\lambda_0 \in (0,1]$ such that for any $n \geq n_0$ and $\lambda \in (0,\lambda_0]$, for any $\tau \in T$,
\begin{equation*}
\gamma_n(p,\omega,\sigma,\tau) \geq v^*(p,\omega)-\varepsilon \quad \text{and} \quad \gamma_\lambda(p,\omega,\sigma,\tau) \geq v^*(p,\omega)-\varepsilon. 
\end{equation*}

\end{theorem} 

\subsection{Stochastic games with long duration and the Mertens conjectures} \label{subsec:mertens}
Let us explain how the results stated in the previous subsection relate to two influential conjectures regarding stochastic games with long duration. 
Mertens, Sorin and Zamir \cite[Chapter IV]{MSZ} have introduced a general model of stochastic game with signals, that we recall briefly. Such a game is described by a finite state space $S$, a finite action set $I$ (resp., $J$) for Player 1 (resp., Player 2), a finite signal set $A$ (resp., $B$) for Player 1 (resp., Player 2), a transition function $\rho:S \times I \times J \rightarrow \Delta(S \times A \times B)$, and a payoff function $g:S \times I \times J \rightarrow \m{R}$. 
Given a prior $p \in \Delta(S \times A \times B)$, the game proceeds as follows: at the outset of the game, a tuple $(s_1,a_1,b_1)$ is drawn according to $p$, and Player 1 (resp., Player 2) is informed of $a_1$ (resp, $b_1$). 
Then, at each stage $m \geq 1$, each player chooses simultaneously some action, denoted by $i_m$ for Player 1 and $j_m$ for Player 2. Stage payoff is $g(s_m,i_m,j_m)$, and a tuple $(s_{m+1},a_{m+1},b_{m+1})$ is selected according to $\rho(s_m,i_m,j_m)$. Player 1 (resp., Player 2) is informed of $a_{m+1}$ (resp., $b_{m+1}$). Strategies, $\lambda$-discounted game, $n$-stage game, $\lambda$-discounted values and $n$-stage values are defined analogously as in Section \ref{subsec:model}. 
\begin{definition}
A stochastic game with signals has a \textit{limit value} if $(v_\lambda)$ and $(v_n)$ converge to the same limit. 
\end{definition}
By definition, when $\lambda$ is small, or $n$ is large, Player 1 can guarantee a quantity close to the limit value. Nonetheless, she may need to know the exact value of $\lambda$ or the exact value of $n$, since optimal strategies may depend in a very sensitive way on these parameters. This motivates the following definition: 
\begin{definition}
Player 1 can \textit{guarantee uniformly} $w$ in $\Gamma(p,q,\omega)$ if for any $\varepsilon>0$, there exists $\sigma \in \Sigma$, $n_0 \geq 1$ and $\lambda_0 \in (0,1]$ such that for any $n \geq n_0$, for any $\lambda \in (0,\lambda_0]$, for any $\tau \in T$,
\begin{equation*}
\gamma_n(p,q,\omega,\sigma,\tau) \geq w-\varepsilon \quad \text{and} \quad \gamma_\lambda(p,q,\omega,\sigma,\tau) \geq w-\varepsilon. 
\end{equation*}
The strategy $\sigma$ is called \textit{$\varepsilon$-uniform optimal strategy}. 
 
The \textit{uniform maxmin} is 
\begin{equation*}
\maxmin(p,q,\omega):=\sup \left\{w \ | \ \text{Player 1 can guarantee uniformly $w$ in $\Gamma(p,q,\omega)$} \right\}. 
\end{equation*}
\end{definition}
Note that when the limit value exists, it follows from the definition that the maxmin is always smaller than the limit value. 
Mertens \cite{M86} conjectured that any stochastic game with signals should have a limit value, and moreover, that when Player 1 is more informed than Player 2, the uniform maxmin and the limit value coincide. By definition, the latter point is equivalent to saying that Player 1 can guarantee uniformly the limit value. 

These conjectures have been proven true in numerous subclasses (see \cite{SZ16} for a survey), but were disproved recently by the author \cite{Z13}.  Despite this counterexample, the question of which stochastic games with signals subclasses satisfy Mertens conjectures remain largely unanswered, and is an active field of research (see e.g. \cite{SV15b,LR20,garrec19}). 
Theorems \ref{main_thm} and \ref{main_thm_2} can now be recasted as follows: 
\begin{center}
\textit{Mertens conjectures hold in absorbing games with incomplete information on both sides}.
\end{center} 
This result and its proof inspire the following new conjecture: 
\\

\emph{
\underline{New conjecture}: Consider a stochastic game with signals such that the state space $S$ can be decomposed as a product $K \times \Omega$, where: 
\begin{itemize}
\item
The state component on $\Omega$ is perfectly observed by each player: formally, for any pair of strategies, for each $m$,
denoting $s_{m}=(k_{m},\omega_{m})$ the state at stage $m$, $\omega_m$ is measurable with respect to $a_m$ and $b_m$. 
\item
for any pair of strategies, for each $m \geq 1$, the following property holds almost surely: 
if $m$ is such that $k_m \neq k_{m+1}$, then for all $m' \geq m+1$, $k_{m'} \neq k_m$. 
\end{itemize}
Then the game satisfies Mertens conjectures: the game has a limit value, and when Player 1 is more informed than Player 2, it is equal to the uniform maxmin.
}
\\

There are several motivations for this statement. First, many subclasses where Mertens conjectures hold satisfy such properties: to name a few, standard stochastic games \cite{BK76,MN81}, repeated games with incomplete information on both sides \cite{MZ71}, repeated games with incomplete information \cite{F82}, absorbing games with a signalling structure \cite{coulomb01}, the model studied in this paper... 

Second, following the counterexample \cite{Z13}, the following informal idea has emerged in the literature \cite{SV15b,LR20}: games satisfying Mertens conjectures feature an \textit{irreversible} property, either on the dynamics, or on the information. The second property of the conjecture stated above has this flavor: the $K$-component has an ``irreversible'' dynamics, in the sense that either it stays where it is, or moves and never goes back. 

Last, the proof approach used in this paper relies heavily on the type of properties stated in the conjecture, and is a good candidate for handling it. We will elaborate on this point in Section \ref{sec:perspectives}. 
We conclude by the following remark:
\begin{remark}
In literature, the definition of uniform maxmin often requires in addition that Player 2 should be able to \textit{defend uniformly} the maxmin, that is:
\begin{eqnarray*}
&\forall\varepsilon>0,& \forall \sigma \in \Sigma, \exists \tau \in T, \ \exists n_0 \geq 1, \forall n \geq n_0, 
\gamma_n(p,q,\omega,\sigma,\tau) \leq \maxmin(p,q,\omega)+\varepsilon 
\\
&\text{and}& \quad \gamma_\lambda(p,q,\omega,\sigma,\tau) \leq \maxmin(p,q,\omega)+\varepsilon.
\end{eqnarray*}
We elaborate on this point in Section \ref{subsec:uniform_maxmin}. 
\end{remark}

\subsection{Organization and insights of the proof}
The proof of Theorem \ref{main_thm} is involved, and divides into three main parts, that correspond to Sections \ref{sec:auxiliary}, \ref{sec:concise} and \ref{sec:coupling}. 

 \paragraph{Section \ref{sec:auxiliary}}Given an absorbing game with incomplete information on both sides, an auxiliary stochastic game with the same discounted values is constructed, where the state corresponds to the triple (belief on Player 1's type, belief on Player 2's type, original state), and action sets are mappings from type sets to original mixed actions. Auxiliary states and actions are perfectly observed by players. 
 State space and actions sets are compact, but stochastic games with compact state space may not have a limit value, even under strong assumptions \cite{Z13}. Hence, 
the idea is to approximate the belief variables (belief on Player 1's type, belief on Player 2's type) by a finitely-valued process, to obtain a stochastic game with finite state space and compact actions sets, where state and actions are still observed. Still, this type of stochastic game may not have a limit value (the first example of this kind is by \cite{vigeral13}, and other examples are available in \cite{Z13,SV15b}). Fortunately, the stochastic game that we obtain is very regular (\textit{semi-algebraic separable stochastic game}), and existence of the limit value holds by \cite{BGV13}. It remains to prove that discounted values of this latter game are close to the original game. This is in fact the main difficulty of the proof, and the object of Sections \ref{sec:concise} and \ref{sec:coupling}. 
\paragraph{Section \ref{sec:concise}}
 Discretizing a compact state space into a finite one is a natural idea, but in the framework of stochastic games, this seldom works (see \cite{NS98,geitner02,V10,MV20} for a few exceptions). Indeed, small errors on the transition function typically propagate when the number of repetitions become large, and makes true state dynamics and approximated state dynamics diverge. A crucial aspect of the discretization made in this paper is that the expectation of the error increment between the true and approximated belief dynamics is 0. As a consequence, the error term between the true dynamics and the approximated dynamics can be bounded by a term that is small compared to the $L^1$-variation of the true belief process. Even though the belief process is a martingale, this $L^1$-variation may not be bounded. Hence, the object of Section \ref{sec:concise} is to prove existence of approximately optimal strategies that generate belief processes with bounded $L^1$-variation (\textit{concise strategies}). 
 \paragraph{Section \ref{sec:coupling}} 
The last part studies the belief dynamics generated by concise strategies, both in the original game and in the approximated game. Using a coupling argument, it shows that both dynamics remain close to each other, and concludes on the proof of Theorem \ref{main_thm}. 
\\

Given Theorem \ref{main_thm}, the proof of Theorem \ref{main_thm_2} is rather simple, and relies again on a coupling argument between the original game and the approximated game. 
\\

Since many of the tools and lemmas used in the proofs are valid in the general framework of stochastic games with incomplete information on both sides, it will be indicated each time whether the absorbing assumption is used or not, in order to make results more accessible for future research.
\section{Auxiliary stochastic games} \label{sec:auxiliary}
\subsection{Stochastic game on the belief set}
We first introduce notations that will be widely used in the sequel. 
Let $\Gamma=(K,L,\Omega,I,J,\rho,g)$ be a stochastic game with incomplete information on both sides. 
An element of $\Delta(K)$ (resp., $\Delta(L)$) will be called a \textit{belief on Player 1's type} (resp., a \textit{belief on Player 2's type}). We denote $X:=\Delta(I)^K$ (resp., $Y:=\Delta(J)^L$) the set of \textit{mixed actions} of Player 1 (resp., Player 2). Let $x \in X$, $y \in Y$, $k \in K$, $\ell \in L$, $i \in I$, and $j \in J$. The quantity $x(i|k)$ represents the probability that Player 1 plays $i$, knowing her type $k$, and $y(j|\ell)$ interprets in the same fashion. 
Denote
\begin{equation*}
\bar{x}^p(i):=\sum_{k \in K} p(k) x(i|k) \quad \text{and} \quad 
\bar{y}^q(j):=\sum_{\ell \in L} q(\ell) y(j|\ell).
\end{equation*}
The quantity $\bar{x}^p(i)$ is the probability that $i$ is played, given that the type prior is $p$ and Player 1 plays $x$. The quantity $\bar{y}^q(j)$ can be interpreted in the same fashion. 
Define $p^x(.|i) \in \Delta(K)$ by
$$
\forall k \in K \quad p^x(k|i) := \left\{
    \begin{array}{ll}
      \frac{x(i|k) p(k)}{\bar{x}^p(i)}  & \mbox{when} \ \bar{x}^p(i) \neq 0 \\
      p(k)  & \mbox{otherwise}
    \end{array}
\right.
$$
and $q^y(.|j) \in \Delta(L)$ by
$$
\forall \ell \in L \quad q^y(\ell|j) := \left\{
    \begin{array}{ll}
       \frac{y(j | \ell) q(\ell)}{\bar{y}^q(j)}  & \mbox{when} \ \bar{y}^p(j) \neq 0 \\
      q(\ell)  & \mbox{otherwise.}
    \end{array}
\right.
$$
The quantity $p^x(.|i)$ is the posterior belief on Player 1's type, given that she played $x$, and the realized action is $i$. The quantity $q^y(.| j)$ is interpreted similarly. 

We associate to $\Gamma$ a \textit{stochastic game} $\Gamma^e$, described by
a state space $\Omega^e:=\Delta(K) \times \Delta(L) \times \Omega$, action set $X$ for Player 1 and $Y$ for Player 2, transition function 
$\rho^e: \Omega^e \times X \times Y \rightarrow \Delta(\Omega^e)$ defined by
$$
\rho^e(p',q',\omega'|p,q,\omega,x,y) := \left\{
    \begin{array}{ll}
      \rho(\omega'|\omega,i,j) 
\bar{x}^p(i) \bar{y}^q(j) & \mbox{when} \ \exists (i,j) \in I \times J, \ (p',q')=(p^x(.|i),q^y(.|j)) \\
      0  & \mbox{otherwise},
    \end{array}
\right.
$$
and payoff function 
\begin{equation*}
g^e(p,q,\omega,x,y)=\sum_{(k,\ell,i,j) \in K \times L \times I \times J} p(k)q(\ell) x(i|k) y(j|\ell) g(k,\ell,\omega,i,j).
\end{equation*}
Given an initial state $\omega^e \in \Omega^e$, the game proceeds as follows: 
\begin{itemize}
\item
At each stage $m \geq 1$, simultaneously, Player 1 chooses $x_m \in X$ and Player 2 chooses $y_m \in Y$. The stage payoff is $g^e(\omega^e_m,x_m,y_m)$. 
\item
A new state $\omega^e_{m+1}$ is drawn according to $\rho^e(\omega^e_{m},x_m,y_m)$, and $(\omega^e_{m+1},x_m,y_m)$ is announced to the players.
\end{itemize} 
A behavior strategy (resp., pure strategy) for Player 1 is a measurable mapping 
\\
$\sigma: \cup_{m \geq 1} (\Omega^e \times X \times Y)^{m-1}
\times \Omega^e
 \rightarrow \Delta(X)$ 
(resp., $\sigma: \cup_{m \geq 1} (\Omega^e \times X \times Y)^{m-1}
\times \Omega^e
 \rightarrow X$). 
A behavior strategy (resp., pure strategy) for Player 2 is a measurable mapping $\tau: \cup_{m \geq 1} (\Omega^e \times X \times Y)^{m-1}
\times \Omega^e \rightarrow \Delta(Y)$
(resp., $\tau: \cup_{m \geq 1} (\Omega^e \times X \times Y)^{m-1}\times \Omega^e \rightarrow Y$). The set of behavior strategies for Player 1 (resp., Player 2) is denoted by $\Sigma^e$ (resp., $T^e$).
A tuple $(\omega^e,\sigma,\tau) \in \Omega^e \times \Sigma^e \times T^e$ induces a probability measure $\m{P}^e_{\omega^e,\sigma,\tau}$ on the set of infinite histories of the game
$(\Omega^e \times X \times Y)^{\m{N}}$, and the expectation with respect to this probability measure is denoted by $\m{E}^e_{\omega^e,\sigma,\tau}$.  
Given $\lambda \in (0,1]$ and $\omega^e \in \Omega^e$, the $\lambda$-discounted game $\Gamma^e_\lambda(\omega^e)$ is the game with payoff 
\begin{equation*}
\gamma^e_\lambda(\omega^e,\sigma,\tau):=\m{E}^e_{\omega^e,\sigma,\tau} \left( \sum_{m \geq 1} \lambda(1-\lambda)^{m-1} g^e(\omega^e_m,x_m,y_m) \right).
\end{equation*}
Since $X$ and $Y$ are compact metric, and $g^e$ and $\rho^e$ are continuous, the value of this game exists \cite[Proposition VII.1.4, p. 394]{MSZ}, and is denoted by
\begin{equation*}
v^e_{\lambda}(\omega^e):=\max_{\sigma \in \Sigma^e} \min_{\tau \in T^e} \gamma^{e}_\lambda(\omega^e,\sigma,\tau)=\min_{\tau \in T^e} \max_{\sigma \in \Sigma^e} \gamma^{e}_\lambda(\omega^e,\sigma,\tau).
\end{equation*}
The $n$-stage game could also be defined analogously to $\Gamma_n$. By Remark \ref{rem:tauberian}, this will not be necessary for our purpose. 

A strategy of Player 1 (resp., Player 2) is \textit{optimal in $\Gamma^e_\lambda$} if it realizes the above left-hand maximum (resp., right-hand side minimum) for any $\omega^e \in \Omega^e$. 
A strategy is \textit{stationary} if at each stage, it plays according to the current state only. Hence, a stationary strategy for Player 1 (resp., Player 2) identifies with a measurable mapping $\sigma: \Omega^e \rightarrow \Delta(X)$
(resp., $\tau: \Omega^e \rightarrow \Delta(Y)$).  

We introduce now notations that be widely used in the remainder of the paper. Given a bounded function $f:\Omega^e \rightarrow \m{R}$ and $(\omega^e,x,y) \in \Omega^e \times X \times Y$, denote
\begin{equation*}
\m{E}^e_{\omega^e,x,y}(f):=\sum_{{\omega^e}' \in \Omega^e} \rho^e({\omega^e}'|\omega^e,x,y)\cdot f({\omega^e}').
 \end{equation*}
Moreover,  given two action sets $A$ and $B$ and a payoff function $h:A \times B \rightarrow \m{R}$, the notation $\val_{(a,b) \in A \times B} \left\{h\right\}$ stands for the value of the zero-sum game $(A,B,h)$, when it exists. 

 The following proposition is a consequence of Mertens, Sorin and Zamir \cite[Proposition IV.3.3, p.186 and Proposition VII.1.4, p.394]{MSZ}: 
\begin{proposition} \label{prop_recursive} \mbox{}
Let $\lambda \in (0,1]$. The following statements hold:
\begin{enumerate}[(i)]
\item
$v^e_{\lambda}=v_\lambda$
\item \label{prop:shapley_e}
The function $v^e_\lambda$ is the only bounded function from $\Omega^e$ to $\m{R}$ that satisfies the \textit{Shapley equation}
\begin{equation} \label{eq:shapley_e}
\forall \omega^e \in \Omega^e, \ v^e_\lambda(\omega^e)=\displaystyle \val_{(x,y) \in X \times Y} \left\{\lambda g^e(\omega^e,x,y)+(1-\lambda) \m{E}^e_{\omega^e,x,y}(v^e_\lambda) \right\}
\end{equation}
 \item \label{prop:stationary_e}
 A stationary strategy $\sigma \in \Sigma^e$ (resp., $\tau \in T^e$) is optimal in $\Gamma^e_\lambda$ if and only if for any $\omega^e \in \Omega^e$, $\sigma(\omega^e)$ (resp., $\tau(\omega^e)$) is optimal in the above game. In particular, each player has a pure optimal stationary strategy in $\Gamma_\lambda^e$. 
\end{enumerate}
\end{proposition}
In the remainder, only pure strategies of $\Gamma^e$ will be considered, and the term ``pure" will be omitted. 
\subsection{Approximation of $\Gamma^e$ by a stochastic game with finite state space} \label{subsec:gamma_f}
Let $\Gamma=(K,L,\Omega,I,J,\rho,g)$ be a stochastic game with incomplete information on both sides, and $\Gamma^e$ be the game defined in the previous section. 
\begin{definition}
A \textit{triangulation of $\Delta(K)$} is a collection $(T_j)_{j \in [1\dots r]}$, $r \geq 1$, of sub-simplices of $\Delta(K)$ such that $\Delta(K)=\cup_{j \in [1 \dots r]} T_j$, and 
for all $j \neq j'$, $T_j \cap T_{j'}$ is either the empty set, or a common face to $T_j$ and $T_{j'}$. 
A triangulation of $\Delta(L)$ is defined similarly. 
\end{definition}
To simplify notations, in the sequel, we will identify a triangulation of $\Delta(K)$ with its set of vertices $P \subset \Delta(K)$, and a triangulation of $\Delta(L)$ with its set of vertices $Q \subset \Delta(L)$. 
\begin{definition}
A triangulation of $\Delta(K) \times \Delta(L)$ is a pair $(P,Q)$, $P \subset \Delta(K)$, $Q \subset \Delta(L)$, such that $P$ is a triangulation of $\Delta(K)$ and $Q$ is a triangulation of $\Delta(L)$. 
\end{definition}
Let $(P,Q)$ be a triangulation of $\Delta(K) \times \Delta(L)$. For $p \in \Delta(K)$, let $T$ be a sub-simplex of the triangulation $P$ that contains $p$, and $p^1,...,p^{|K|} \in \Delta(K)$ the vertices of $T$. Let $a_1(p),a_2(p),...,a_{|K|}(p) \in [0,1]$ be the corresponding coordinates:
\begin{equation*}
p=\sum_{k \in K} a_k(p) \cdot p^k.
\end{equation*}
We introduce similar notations for $q \in \Delta(L)$: 
\begin{equation*}
q=\sum_{\ell \in L} b_{\ell}(q) \cdot q^{\ell}.
\end{equation*}
We denote by $S[p] \in \Delta(P)$ the \textit{splitting} of $p$, that is, the law on $P$ defined by
\begin{equation*}
S[p]:=\sum_{k \in K} a_k(p) \cdot \delta_{p^k}.
\end{equation*}
We will denote $S[p'|p]$ the probability of $p'$ under law $S[p]$: hence, when $p'=p^k, \ k \in K$, $S[p'|p]=a_k(p)$, and $S[p'|p]=0$ otherwise. The splitting of $q$, denoted by $S[q] \in \Delta(Q)$, is defined similary, along with the notation $S[q'|q]$.
Thus, $S$ is defined on $\Delta(K) \cup \Delta(L)$, and $S[\Delta(K)] \subset \Delta(P)$ and $S[\Delta(L)] \subset \Delta(Q)$. 

Define a stochastic game with state space $\Omega^f:=P \times Q \times \Omega$, action sets $X$ for Player 1 and $Y$ for Player 2, transition function $\rho^f$ such that for all $(p,q,\omega,x,y) \in \Omega^f \times X \times Y$ and $(p',q',\omega') \in \Omega^f$,
\begin{equation*}
\rho^f(p',q',\omega'|p,q,\omega,x,y):=\sum_{i \in I} \sum_{j \in J}  
\rho^e(p^x(.|i),q^y(.|j),\omega'|p,q,\omega,x,y)  S[p'|p^x(.|i)] S[q'|q^y(.|j)],
\end{equation*}
and payoff function $g^f$ defined by
\begin{equation*}
g^f:={g^e}_{|\Omega^f \times X \times Y}.
\end{equation*}
The game $\Gamma^f$ proceeds in the same way as $\Gamma^e$, up to substituting $\Omega^f$ to $\Omega^e$, $\rho^f$ to $\rho^e$, and $g^f$ to $g^e$. Note that the definition of $\Gamma^f$ depends on $\Gamma^e$, but also on the choice of the triangulation $(P,Q)$. To avoid heavy notations, such a dependence is omitted. 
%
%
%
%
    
The following picture represents an example of belief evolution in $\Delta(K)$, when $K=\left\{k_1,k_2,k_3\right\}$, $I=\left\{A,B\right\}$, and prior belief is $p=\frac{1}{3} \cdot \delta_{k_1}+\frac{1}{3} \cdot \delta_{k_2}+\frac{1}{3} \cdot \delta_{k_3}$. The two arrows indicate the decomposition of $p$ into two posterior beliefs $p^x(.|A)$ and $p^x(.|B)$, under some mixed action $x \in \Delta(\left\{A,B\right\})^{\left\{k_1,k_2,k_3\right\}}$. This corresponds to the belief dynamics in $\Gamma^e$. In $\Gamma^f$, these beliefs are splitted again on the vertices of the corresponding sub-simplex, namely $p^1,p^2,p$ when $A$ is played, and $p,p^3,p^4$ when $B$ is played (see the dashed lines). 
\begin{center} 
\includegraphics[scale=0.4]{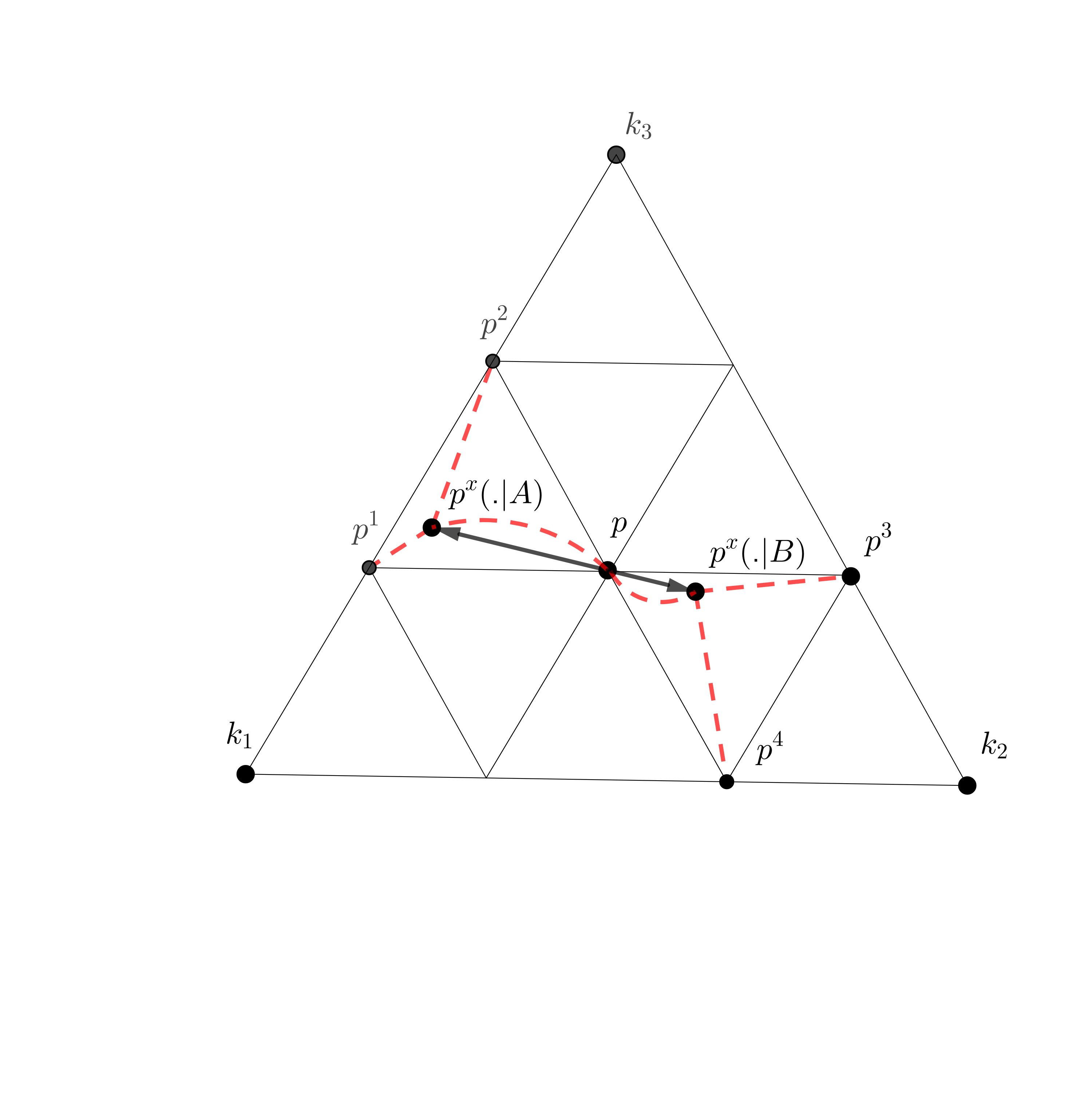}
\end{center}
We denote by $\Sigma^f$ and $T^f$ the set of behavior strategies 
in $\Gamma^f$, and $\Gamma^f_\lambda$ the $\lambda$-discounted game. By \cite[Proposition VII.1.4, p. 394]{MSZ}, $\Gamma^f_\lambda$ has a value, which is denoted by $v^f_\lambda$. 
Given a bounded function $h:\Omega^f \rightarrow \m{R}$ and $(\omega^f,\mu,\nu) \in \Omega^f \times \Delta(X) \times \Delta(Y)$, denote
\begin{equation*}
\m{E}^f_{\omega^f,\mu,\nu}(h):=\sum_{{\omega^f}' \in \Omega^f} \int_{X \times Y} \rho^f({\omega^f}'|\omega^f,x,y)\cdot h({\omega^f}') d\mu(x) d\nu(y).
 \end{equation*}
 We extend $g^f$ linearly by the formula
 \begin{equation*}
 \forall (\omega^f,\mu,\nu) \in \Omega^f \times \Delta(X) \times \Delta(Y),
 \quad g^f(\omega^f,\mu,\nu):=\int_{X \times Y} g^f(\omega^f,x,y) d\mu(x) d\nu(y).
 \end{equation*}

The following proposition is a consequence of Mertens, Sorin and Zamir \cite[Proposition IV.3.3, p.186 and Proposition VII.1.4, p.394]{MSZ}: 
\begin{proposition} \label{prop_recursive} \mbox{}
\begin{enumerate}[(i)]
\item \label{prop:shapley_f}
The function $v^f_\lambda$ is the only function from $\Omega^f$ to $\m{R}$ that satisfies the \textit{Shapley equation}
\begin{equation} 
\forall \omega^f \in \Omega^f, \ v^f_\lambda(\omega^f)=\displaystyle \val_{(\mu,\nu) \in \Delta(X) \times \Delta(Y)} \left\{\lambda g^f(\omega^f,\mu,\nu)+(1-\lambda) \m{E}^f_{\omega^f,\mu,\nu}(v^f_\lambda) \right\}
\end{equation}
 \item \label{prop:stationary_f}
 A stationary strategy $\sigma \in \Sigma^f$ (resp., $\tau \in T^f$) is optimal in $\Gamma^f_\lambda$ if and only if for any $\omega^f \in \Omega^f$, $\sigma(\omega^f)$ (resp., $\tau(\omega^f)$) is optimal in the above game. In particular, each player has a (behavior) optimal stationary strategy in $\Gamma_\lambda^f$. 
\end{enumerate}
\end{proposition}
\begin{remark}
It is unclear whether $\Gamma^f_\lambda$ has a value in pure strategies or not. 
\end{remark}
Let us now turn to the question of existence of the limit value in $\Gamma^f$. In general, stochastic games with finite state space and compact action sets may not have a limit value \cite{vigeral13}. Fortunately, $\Gamma^f$ has the following strong property:
   \begin{proposition}
The game $\Gamma^f$ is a stochastic game with separable semi-algebraic transition and payoff functions, in the sense of \cite[Theorem 3]{BGV13}:  there exists finite sets $I'$, $J'$, real numbers $m^{\omega^f}_{i' j'}, \ (\omega^f,i',j') \in \Omega^f \times I' \times J'$,  and functions $(\omega^f,x) \rightarrow a_{i'}(\omega^f,x), \ i' \in I'$ and $(\omega^f,y) \rightarrow b_{j'}(\omega^f,y), \ j' \in J$ that are continuous and semi-algebraic in the second variable, such that
\begin{equation*}
\forall \omega^f \in \Omega^f, \ \forall x \in X, \  \forall y \in Y, \quad g^f(\omega^f,x,y)=\sum_{i' \in I'} \sum_{j' \in J'} m^{\omega^f}_{i'j'} a_{i'}(\omega^f,x) b_{j'}(\omega^f,y),
\end{equation*}
and there exists real numbers $n^{\omega^f,{\omega^f}'}_{i,j}$, $(\omega^f,{\omega^f}',i,j) \in ({\Omega^f})^2 \times I \times J$, and functions 
$(\omega^f,{\omega^{f}}',x) \rightarrow c_{i}(\omega^f,{\omega^{f}}',x), \ i \in I$ and $(\omega^f,{\omega^f}',y) \rightarrow d_{j}(\omega^f,{\omega^f}',y), \ j \in J$ that are continuous and semi-algebraic in the third variable, such that
\begin{equation*}
\forall (\omega^f,{\omega^f}') \in ({\Omega^f})^2, \ \forall x \in X, \forall y \in Y, \quad \rho^f({\omega^f}'|\omega^f,x,y)=\sum_{i \in I} \sum_{j \in J} n^{\omega^f,{\omega^f}'}_{i,j} c_i(\omega^f,{\omega^f}',x)d_j(\omega^f,{\omega^f}',y). 
\end{equation*}
\end{proposition}
\begin{proof}
Set $I':=K \times I$ and $J':=L \times J$. 
For all $\omega^f=(p,q,\omega) \in \Omega^f$, $i'=(k,i) \in I'$, $j'=(\ell,j) \in J'$, define $a_{i'}(\omega^f,x):=x(i|k)$, $b_{j'}(\omega^f,y):=y(j|\ell)$, and $m^{\omega^f}_{i'j'}:=p(k)q(\ell)g(k,\ell,\omega,i,j)$. For all $(i',j') \in I' \times J'$, the mappings 
$a_{i'}$ and $b_{j'}$ are continuous and semi-algebraic in the second variable, hence satisfy the conditions of the proposition.  

Moreover,
for all ${\omega^f}=(p,q,\omega) \in \Omega^f$, ${\omega^f}'=(p',q',\omega') \in \Omega^f$, $(i,j) \in I \times J$, define $n_{i,j}^{\omega^f,{\omega^f}'}:=\rho({\omega}'| \omega,i,j)$, $c_i(\omega^f,{\omega^f}',x):=\bar{x}^p(i)S[p'|p^x(.|i)]$, $d_j(\omega^f,{\omega^f}',y)=\bar{y}^q(j) S[q'|q^y(.|j)]$. Let $i \in I$. The mapping $c_i$ is continuous in the third variable. Let $p \in P$. 
The mapping $x \rightarrow \bar{x}^p(i)$ is linear, hence semi-algebraic on $X$. Thus, the set $A:=\left\{x \in X \ | \ \bar{x}^p(i) \neq 0 \right\}$ is semi-algebraic. Moreover, for $p' \in P$, the mapping $x \rightarrow S[p'|p^x(.|i)]$ is semi-algebraic on $A$. Indeed, let $T_1,T_2,\dots,T_{r_0}$ be the set of sub-simplices of $P$ that have $p'$ as a vertex. Let $s \in [1 \dots r_0]$ and let $p^1,\dots,p^{|K|-1}$ the vertices of $T_s$ different from $p'$. Let $A_s:=\left\{x \in A \ | \ p^x(.|i) \in T_s \right\}$, which is a semi-algebraic set. Let $p^0$ that is orthogonal to each $p^1,\dots,p^{|K|-1}$. Then 
\begin{equation*}
S[p'|p^x(.|i)]=\frac{p^x(.|i) \cdot p^0}{p' \cdot p^0},
\end{equation*}
which implies that the mapping $x \rightarrow S[p'|p^x(.|i)]$ is semi-algebraic on $A_s$, hence semi-algebraic on $A:=\cup_{1 \leq s \leq r_0} A_s$. Since $c_i(.,x)=0$ for $x \notin A$, we deduce that $c_i$ is 
semi-algebraic in the third variable. Similarly, for all $j \in J$, $d_j$ is continuous and semi-algebraic in the third variable. Hence, $\Gamma^f$ is stochastic game with separable semi-algebraic transition and payoff functions.
\end{proof}
It turns out that semi-algebraic separable stochastic games have a limit value, thanks to \cite[Theorem 3]{BGV13}, hence:
\begin{proposition} \label{asympt_f}
$\Gamma^f$ has a limit value. In particular,  $(v^f_\lambda)$ converges, as $\lambda$ tends to 0. 
\end{proposition} 
In order to prove Theorem \ref{main_thm}, we will thus prove that when the sub-simplices of the triangulation ``go to 0'', the distance between $v^f_\lambda$ and $v^e_\lambda$ vanishes. Formalizing this idea requires the following definition.
\begin{definition} 
The \textit{stepsize} of a triangulation $P \subset \Delta(K)$ is the maximum euclidean distance between two neighbor vertices.  
Let $\alpha>0$ and $C>0$. An $(\alpha,C)$-triangulation $P \subset \Delta(K)$ is a triangulation with stepsize $s$ smaller than $\alpha$, and such that for any $p \in \Delta(K)$, for any vertex $p' \in P$ such that $S[p'|p]>0$, 
\begin{equation} \label{eq:tr}
1-S[p'|p] \leq \frac{C}{s} \left\|p'-p\right\|_2.
\end{equation}
An $(\alpha,C)$-triangulation of $\Delta(L)$ is defined similarly, and a pair $(P,Q)$ such that $P$ and $Q$ are respectively $(\alpha,C)$-triangulations of $\Delta(K)$ and $\Delta(L)$ will be called an $(\alpha,C)$ triangulation of $\Delta(K) \times \Delta(L)$. 
\end{definition}
Intuitively, equation \eqref{eq:tr} means that the heights of each sub-simplex of the triangulation are of the same order as the stepsize of the triangulation: hence, sub-simplices are not ``flat''. 


\begin{proposition} \label{prop:ex_tr}
There exists $C>0$ such that for any $\alpha>0$, there exists an $(\alpha,C)$-triangulation of $\Delta(K) \times \Delta(L)$.
\end{proposition}
\begin{proof}
If $|K|=1$, then the left-hand side of inequality \eqref{eq:tr} is 0, hence the proposition holds. Assume $|K|>1$, and let $d:=|K|-1$. 
Let $N \geq 1$. By \cite{EG00}, there exists a triangulation $P$ of $\Delta(K)$ composed of $N^d$ sub-simplices with the same volume $V/N^d$, where $V$ is the volume of $\Delta(K)$. Moreover, the stepsize $s$ is smaller than $\sqrt{2}d/N$ \cite[p.711]{EG00}. Let $p \in \Delta(K)$ and $p' \in P$ such that $a:=S[p'|p]>0$. Let $T$ be a sub-simplex that contains $p$ and $p'$. Decompose $p$ as $p=a \cdot p'+(1-a) \cdot p''$, where $p'' \in \Delta(K)$ lies in the facet $F$ of $T$ that is opposite to $p'$. Then $\left\|p'-p\right\|_2=(1-a) \left\|p'-p''\right\|_2$. Note that $\left\|p'-p''\right\|_2 \geq d_2(p',F)$, where $d_2(p',F)$ is the euclidean distance between $p'$ and the set $F$. This yields
\begin{equation} \label{eq:A_tr}
 1-a \leq \left\|p'-p\right\|_2/d_2(p',F).
 \end{equation}
 Let us give a lower bound on $d_2(p',F)$. The $d$-dimensional volume of $T$ is by definition equal to  $d_2(p',F) \cdot V'/d$, where $V'$ is the $(d-1)$-dimensional volume of $F$. We deduce that
$V/N^d=d_2(p',F) \cdot V'/d$, and $d_2(p',F)=d V /(V'N^d)$. A rough majorization of $V'$ is that it should be smaller than the volume of a $(d-1)$-dimensional hypercube with side length $\sqrt{2} d/N$, hence 
$V' \leq (\sqrt{2} d/N)^{d-1}$. 
We deduce that $d_2(p',F) \geq d V  / (N (\sqrt{2}d)^{d-1})$. Setting $C:=(\sqrt{2}d)^{d}/dV$ and plugging the inequality in \eqref{eq:A_tr} gives
\begin{equation*}
1-a \leq C (\sqrt{2}d/N)^{-1} \left\|p'-p\right\|_2 \leq \frac{C}{s} \left\|p'-p\right\|_2. 
\end{equation*}
This implies that the triangulation $P$ is a $(\sqrt{2}d/N,C)$-triangulation. A similar construction can be made for the simplex $\Delta(L)$. Since $N$ is arbitrary, this yields the proposition. 
\end{proof}
A technical difficulty met in the proof of Theorem \ref{main_thm} is that it is not clear whether $v^f_\lambda$ is a Lipschitz function with respect to $p$ and $q$ or not. To remedy this problem, we introduce the following assumption: 
\begin{assumption} \label{asslip}
Assume that there exists $i^* \in I, j^* \in J$ and $\omega^1, \omega^2 \in \Omega$ two absorbing states such that $g(\omega^1,.)=-\left\|g\right\|_{\infty}$, $g(\omega^2,.)=\left\|g\right\|_{\infty}$, and for all non-absorbing state $\omega \in \Omega$, for all $j \neq j^*$, $\rho(\omega,i^*,j)=\delta_{\omega^1}$, and for all $i \neq i^*$, $\rho(\omega,i,j^*)=\delta_{\omega^2}$, and $\rho(\omega,i^*,j^*)=\frac{1}{2} \cdot \delta_{\omega^1}+\frac{1}{2} \cdot \delta_{\omega^2}$. 
\end{assumption}

As far as Theorem \ref{main_thm} is concerned, this assumption is without loss of generality. Indeed, one can always add to $I$ and $J$ actions $i^*$ and $j^*$, and add to $\Omega$ states $\omega^1$ and  $\omega^2$ that satisfy the above assumption. In $\Gamma_\lambda$, any strategy that plays $i^*$ or $j^*$ with positive probability after some non-zero probability history is clearly dominated, hence adding $i^*$ and $j^*$ does not change 
$v_\lambda=v^e_{\lambda}$. Still, it could change $v^f_\lambda$ (again, it is not clear whether this is the case or not). Such an assumption will turn useful to prove a Lipschitz-type property on $v^f_\lambda$ with respect to $p$ and $q$ (it will be stated later on, in Subsection \ref{subsec:modif}). This Lipschitz property is indispensable to prove the following proposition:
\begin{proposition} \label{comparison_e_f} 
Let $\Gamma=(K,L,\Omega,I,J,\rho,g)$ be an absorbing game with incomplete information on both sides, that satisfies Assumption \ref{asslip}. Then for all $\varepsilon>0$, for all $C>0$, there exists $\alpha>0$ such that for any $(\alpha,C)$-triangulation $(P,Q)$, 
for all $\omega^f \in \Omega^f$ and $\lambda \in (0,1]$, 
\begin{equation*}
\left|v^e_\lambda(\omega^f)-v^f_\lambda(\omega^f)\right| \leq \varepsilon.
\end{equation*}
\end{proposition}
Proving this proposition is the main difficulty to show Theorem \ref{main_thm}. Indeed, let us prove that it readily implies Theorem \ref{main_thm}:
\begin{proof}[Proof of Theorem \ref{main_thm} admitting Proposition \ref{comparison_e_f}]
Let $\Gamma=(K,L,\Omega,I,J,\rho,g)$ be an absorbing game with incomplete information on both sides. 
As explained before, as far as the proof of Theorem \ref{main_thm} is concerned, we can assume w.l.o.g. that Assumption \ref{asslip} is in force. 
\\ 
Let $C>0$ given by Proposition \ref{prop:ex_tr}. Let $\varepsilon>0$. Let $\alpha$ be given by Proposition \ref{comparison_e_f}, and $\alpha_0:=\min(\alpha,\varepsilon)$. By Proposition \ref{prop:ex_tr},  there exists an $(\alpha_0,C)$-triangulation $(P,Q)$ of $\Delta(K) \times \Delta(L)$. Because $\alpha_0 \leq \alpha$, $(P,Q)$ is an $(\alpha,C)$-triangulation. By Proposition \ref{comparison_e_f}, for all $\lambda \in (0,1]$, for all $(p,q,\omega) \in \Omega^f$,
\begin{equation*}
v^f_\lambda(p,q,\omega)-\varepsilon \leq v^e_\lambda(p,q,\omega) \leq v^f_\lambda(p,q,\omega)+\varepsilon.
\end{equation*}
Let $(p,q,\omega) \in \Omega^f$, and $v^*(p,q,\omega):=\lim_{\lambda \rightarrow 0} v^f_\lambda(p,q,\omega)$, which exists by Proposition \ref{asympt_f}. By the previous inequality, we deduce that
\begin{equation*}
\limsup_{\lambda \rightarrow 0} v^e_\lambda(p,q,\omega) \leq v^*(p,q,\omega)+\varepsilon \leq  \liminf_{\lambda \rightarrow 0} v^e_\lambda(p,q,\omega)+2 \varepsilon.
\end{equation*}
Because $v^e_\lambda(\omega,p,.)$ and $v^e_\lambda(\omega,.,q)$ are $\left\|g\right\|_{\infty}$-Lipschitz, and $\alpha_0 \leq \varepsilon$, we deduce that for all $(p,q,\omega) \in \Omega^e$, 
\begin{equation*}
\limsup_{\lambda \rightarrow 0} v^e_\lambda(p,q,\omega) \leq \liminf_{\lambda \rightarrow 0} v^e_\lambda(p,q,\omega)+2 \varepsilon+2\varepsilon \left\|g\right\|_{\infty}.
\end{equation*}
Since $\varepsilon$ is arbitrary, we get that $(v^e_\lambda)$ converges pointwise. Because $(v^e_\lambda)$ is equi-Lipschitz, it converges uniformly. By Proposition \ref{prop_recursive}, $v^e_\lambda=v_\lambda$, hence $(v_\lambda)$ converges uniformly. By Remark \ref{rem:tauberian}, this proves Theorem \ref{main_thm}.  
\end{proof}
Notice that up to switching the roles of Player 1 and Player 2, to prove Proposition \ref{comparison_e_f}, it is enough to show the following:
\begin{proposition} \label{comparison_e_f_w}
Let $\Gamma=(K,L,\Omega,I,J,\rho,g)$ be an absorbing game with incomplete information on both sides, that satisfies Assumption \ref{asslip}.
Then for all $\varepsilon>0$, for all $C>0$, there exists $\alpha>0$ such that for any $(\alpha,C)$-triangulation $(P,Q)$, for all $\omega^f \in \Omega^f$ and $\lambda \in (0,1]$, 
\begin{equation*}
v^f_\lambda(\omega^f) \geq v^e_\lambda(\omega^f)-\varepsilon.
\end{equation*}
\end{proposition}

The next two sections are dedicated to the proof of this proposition. We will build a certain type of approximately optimal strategies for Player 1 in $\Gamma^e_\lambda$, and ``copy" them in $\Gamma^f_\lambda$. Hence, most concepts and lemmas will be stated under Player 1's perspective. 
\section{Belief martingales with bounded variation} \label{sec:concise}
Taking aside technical details, the goal of this section is to build $\varepsilon$-optimal strategies for Player 1 in $\Gamma^e$ such that the belief process on her type $(p_m)_{m \geq 1}$ has a $L^1$-variation 
$\sum_{m \geq 1} \left\|p_{m+1}-p_m \right\|_1$ that can be bounded in expectation by a term depending only on $\varepsilon$ and on the data of the game. Such a property will be crucial in the coupling between $\Gamma^e$ and $\Gamma^f$ made in Section \ref{sec:coupling}. 
\subsection{Main result of the section and notion of concise strategy}
Let $\Gamma=(K,L,\Omega,I,J,\rho,g)$ be a stochastic game with incomplete information on both sides.
\begin{definition}
Let $\varepsilon>0$. A mixed action $x \in \Delta(I)^K$ is \textit{$\varepsilon$-ambiguous} if it satisfies the following property:
For all $k \in K$ and $i \in I$, 
\begin{equation*}
p^x(k|i) \geq \varepsilon \min_{k' \in K} p(k'). 
\end{equation*}
A strategy is \textit{$\varepsilon$-ambiguous} if for any history $h$, $\sigma(h)$ is $\varepsilon$-ambiguous. 
\end{definition}
\begin{definition}
Let $\varepsilon \geq 0$. The \textit{$\varepsilon$-frontier} of $\Delta(K)$ is the set $F_{\varepsilon} \subset \Delta(K)$ defined by 
\begin{equation*}
F_\varepsilon:=\left\{p \in \Delta(K) \ | \exists k \in K, \  p(k) \leq \varepsilon \right\}.
\end{equation*}
\end{definition}
The aim of Section \ref{sec:concise} is to prove the following result:
\begin{proposition} \label{prop:bounded_var}
Let $\Gamma=(K,L,\Omega,I,J,\rho,g)$ be an absorbing game with incomplete information on both sides.
Let $\varepsilon \in (0,1/4]$, $T:=\max\left\{m \geq 1, p_m \in \Delta(K) \setminus F_{\varepsilon} \right\}$, and $\lambda \in (0,1]$. Then Player 1 has a (pure) $12\varepsilon \left\|g\right\|_\infty$-optimal stationary strategy $\sigma$ in $\Gamma^e_\lambda$ that is $\varepsilon$-ambiguous, and such that for all $\tau \in T^e$,
\begin{equation*}
\m{E}^{e}_{\omega^0,\sigma,\tau}\left(\sum_{m=1}^{T} \left\|p_{m+1}-p_m\right\|_1\right) \leq 3 \sqrt{|K|} \varepsilon^{-5}. 
\end{equation*}
\end{proposition}
The reader that wants first to know more about the role held by the above proposition in the proof of Theorem \ref{main_thm} can jump directly to Section \ref{sec:coupling}. 

The $\varepsilon$-ambiguous property is rather easy to obtain, and the difficult part is to ensure the $L^1$-bounded variation condition. With regards to the latter

Again, in what follows, a general stochastic game with incomplete information on both sides
$\Gamma=(K,L,\Omega,I,J,\rho,g)$  is considered, and the absorbing assumption is precised when needed. 
\begin{definition}
Let $\varepsilon>0$, and $x \in \Delta(I)^K$. An action $i \in I$ is \textit{$(x,\varepsilon)$-non-revealing} at $p$ if $\bar{x}^p(i) \neq 0$ and for all $k \in K$, we have
\begin{equation*}
(1-\varepsilon) \bar{x}^p(i) \leq x(i|k) \leq (1+\varepsilon) \bar{x}^p(i). 
\end{equation*}
An action such that $\bar{x}^p(i) \neq 0$ and $i$ is not $(x,\varepsilon)$-non-revealing at $p$ is called \textit{$(x,\varepsilon)$-revealing at $p$}. 
The set of $(x,\varepsilon)$-non-revealing actions at $p$ is denoted by $NR[x,\varepsilon,p]$, and the set of $(x,\varepsilon)$-revealing actions at $p$ is denoted by $R[x,\varepsilon,p]$.
\end{definition}
It is well-known that bounded martingales have bounded $L^2$-variation, hence can not make too many significant ``jumps": formally, for each $\varepsilon>0$ and $(\omega^e, \sigma, \tau) \in \Omega^e \times \Sigma^e \times T^e$, 
\\
$\m{E}^e_{\omega^e,\sigma,\tau}\left(\sum_{m=1}^{+\infty} 1_{\left\|p_{m+1}-p_m\right\| \geq \varepsilon} \left\|p_{m+1}-p_m\right\|_1\right)$ is bounded independently of $(\omega^e,\sigma,\tau)$. 
Hence, to obtain bounded $L^1$-variation, a first naive attempt would be to consider a set of mixed actions $x$ for Player 1 such that either $x$ reveals a significant amount of information about the type, or it does not reveal any: this would lead to the condition ``for all $i \in NR[x,\varepsilon,p]$ and $k,k' \in K$, $x(i|k)=x(i|k')$", where $p$ is the current type belief. Unfortunately, such a definition is too demanding, and it may be that all mixed actions in this set are significantly suboptimal. Instead, we will consider mixed actions $x$ such that \textit{conditional to $i \in NR[x,\varepsilon,p]$}, the probability of playing $i$ does not depend on $k$. 


This leads to the following definition: 
\begin{definition} \label{def_concise}
Let $\varepsilon>0$. 
A mixed action $x \in X$ is $\varepsilon$-\textit{concise} at $p \in \Delta(K)$ if it satisfies the following property:
for all $i \in NR[x,\varepsilon,p]$, for all $k \in K$,
\begin{equation*}
x(i|k)=\frac{x(NR[x,\varepsilon,p]|k)}{\bar{x}^p(NR[x,\varepsilon,p])} \bar{x}^p(i).
\end{equation*}
\end{definition}
\begin{definition}
Let $\varepsilon >0$. 
A strategy $\sigma$ of Player 1 is \textit{$\varepsilon$-concise} if for any history $h=(h',p,q,\omega)$, $\sigma(h)$ is $\varepsilon$-concise at $p$. 
\end{definition}
\begin{remark}
A stationary strategy $\sigma: \Omega^e \rightarrow X$ is $\varepsilon$-concise if for any $(p,q,\omega) \in \Omega^e$, $\sigma(p,q,\omega)$ is $\varepsilon$-concise at $p$. 
\end{remark}
The next subsection proves existence of approximately optimal strategies for Player 1 that are $\varepsilon$-concise and $\varepsilon$-ambiguous, in the absorbing framework. The last subsection proves that such strategies generate a belief process $(p_m)$ that satisfies the inequality of Proposition \ref{prop:bounded_var}, even without the absorbing assumption. Combining both results prove Proposition \ref{prop:bounded_var}.
\subsection{Existence of approximately optimal concise strategies}	
The main result of this section is the following proposition:
\begin{proposition} \label{prop:opt_concise}
Let $\Gamma=(K,L,\Omega,I,J,\rho,g)$ be an absorbing game with incomplete information on both sides. Let $\varepsilon \in(0,1/4]$ and $\lambda \in (0,1]$. There exists a $12 \varepsilon \left\| g \right\|_\infty$-optimal stationary strategy in $\Gamma^e_\lambda$ 
that is $\varepsilon$-concise and $\varepsilon$-ambiguous. 
\end{proposition}
Again, a general stochastic game with incomplete information on both sides $\Gamma=(K,L,\Omega,I,J,\rho,g)$ is considered, and the absorbing assumption is precised when needed. 

The idea of the proof is to transform a stationary optimal strategy of Player 1 in $\Gamma^e_\lambda$ into an $\varepsilon$-concise and $\varepsilon$-ambiguous stationary strategy that is  $12 \varepsilon\left\|g\right\|_{\infty}$-optimal in $\Gamma^e_\lambda$. Such a transformation relies on the following object:
\begin{definition}
Let $\varepsilon \in (0,1]$. The \textit{$\varepsilon$-silent mapping} is the function $c_{\varepsilon}: X \times \Delta(K) \rightarrow X$ defined by:
$$
\forall (x,p) \in X \times \Delta(K), \quad [c_\varepsilon(x,p)](i|k) := \left\{
    \begin{array}{ll}
       \left[1-\varepsilon) \frac{x(NR[x,\varepsilon,p]|k)}{\bar{x}^p(NR[x,\varepsilon,p])} +\varepsilon \right]\bar{x}^p(i) & \mbox{when} \ i \in NR[x,\varepsilon,p]  \\
       (1-\varepsilon) x(i|k)+\varepsilon \bar{x}^p(i) & \mbox{otherwise.}
    \end{array}
\right.
$$
\end{definition}
The first property of this mapping is to preserve action distributions: 
\begin{proposition} \label{prop_conservation}
Let $\varepsilon \in (0,1]$, $p \in \Delta(K)$, $x \in X$ and $x':=c_{\varepsilon}(x,p)$. Then for all $i \in I$, $\overline{x'}^p(i)=\bar{x}^p(i)$.
\end{proposition}
\begin{proof}
Assume $i \in NR[x,\varepsilon,p]$. We have 
\begin{eqnarray*}
\bar{x'}^p(i)&=&\sum_{k \in K}  \left[(1-\varepsilon) \frac{x(NR[x,\varepsilon,p]|k)}{\bar{x}^p(NR[x,\varepsilon,p])} +\varepsilon \right]\bar{x}^p(i) p(k)
\\
&=&\left[1-\varepsilon+\varepsilon \right] \bar{x}^p(i)=\bar{x}^p(i). 
\end{eqnarray*}
Assume $i \in I \setminus NR[x,\varepsilon,p]$. We have
\begin{eqnarray*}
\bar{x'}^p(i)=\sum_{k \in K} x'(i|k) p(k)=\sum_{k \in K}[(1-\varepsilon) x(i|k)+\varepsilon \bar{x}^p(i)]p(k)=(1-\varepsilon) \bar{x}^p(i)+\varepsilon  \bar{x}^p(i)=\bar{x}^p(i).
\end{eqnarray*}
\end{proof}
The second property states that these mappings transform mixed actions into $\varepsilon$-concise and $\varepsilon$-ambiguous mixed actions: 
\begin{proposition} \label{prop_concise_mapping}
Let $\varepsilon \in (0,1/4]$ and $\varepsilon_0:=\frac{1-\sqrt{1-4\varepsilon}}{2}$.   Let $p \in \Delta(K)$, $x \in X$ and $x':=c_{\varepsilon_0}(x,p)$. Then $NR[x',\varepsilon,p] = NR[x,\varepsilon_0,p]$, and moreover, $x'$ is $\varepsilon$-concise and $\varepsilon$-ambiguous at $p$. 
\end{proposition}
\begin{proof}
Let us first prove that $NR[x',\varepsilon,p] \subset NR[x,\varepsilon_0,p]$, which is equivalent to $R[x,\varepsilon_0,p] \subset R[x',\varepsilon,p]$. Indeed, let $i \in R[x,\varepsilon_0,p]$, hence $x'(i|k)=(1-\varepsilon_0) x(i|k)+\varepsilon_0 \bar{x}^p(i)$. There exists $k \in K$ such that either 
\begin{equation*}
x(i|k) < (1-\varepsilon_0) \bar{x}^p(i) \quad \text{or} \quad x(i|k) > (1+\varepsilon_0) \bar{x}^p(i),
\end{equation*}
which implies
\begin{equation*}
x'(i|k) < [(1-\varepsilon_0)^2+\varepsilon_0] \bar{x}^p(i) \quad \text{or} \quad x'(i|k) > [(1-\varepsilon_0)(1+\varepsilon_0)+\varepsilon_0] \bar{x}^p(i). 
\end{equation*}
Given that $(1-\varepsilon_0)^2+\varepsilon_0=1-\varepsilon$ and $(1-\varepsilon_0)(1+\varepsilon_0)+\varepsilon_0=1+\varepsilon$, this implies
\begin{equation*}
x'(i|k) < (1-\varepsilon) \bar{x}^p(i) \quad \text{or} \quad x'(i|k) > (1+\varepsilon) \bar{x}^p(i),
\end{equation*}
which yields $i \in R[x',\varepsilon,p]$. Conversely, let $i \in NR[x,\varepsilon_0,p]$. We have
\begin{eqnarray*}
x'(i|k)&=&\left[(1-\varepsilon_0) \frac{x(NR[x,\varepsilon_0,p]|k)}{\bar{x}^p(NR[x,\varepsilon_0,p])} +\varepsilon_0 \right]\bar{x}^p(i)
\\
&\leq& \left[(1-\varepsilon_0)(1+\varepsilon_0) +\varepsilon_0 \right] \bar{x}^p(i)
\\\
&=& (1+\varepsilon) \bar{x}^p(i)=(1+\varepsilon) \overline{x'}^p(i).
\end{eqnarray*}
Similarly, we have $x'(i|k) \geq \left[(1-\varepsilon_0)(1-\varepsilon_0) +\varepsilon_0 \right] \bar{x}^p(i)=(1-\varepsilon) \overline{x'}^p(i)$, and $i \in NR[x',\varepsilon,p]$. We have thus proven that 
$NR[x',\varepsilon,p] = NR[x,\varepsilon_0,p]$. 

Let us prove that $x'$ is $\varepsilon$-concise at $p$. 
Let $i \in NR[x',\varepsilon,p]=NR[x,\varepsilon_0,p]$, we have
\begin{eqnarray*}
\frac{x'(NR[x',\varepsilon,p]|k)}{\overline{x'}^p(NR[x',\varepsilon,p])} \bar{x'}^p(i)&=& \frac{x'(NR[x',\varepsilon,p]|k)}{\overline{x}^p(NR[x',\varepsilon,p])} \bar{x}^p(i)\\
&=& \frac{\sum_{i' \in NR[x',\varepsilon,p]}  \left[(1-\varepsilon_0) \frac{x(NR[x,\varepsilon_0,p]|k)}{\bar{x}^p(NR[x,\varepsilon_0,p])} +\varepsilon_0\right] \bar{x}^p(i')}{\bar{x}^p(NR[x',\varepsilon,p])}\bar{x}^p(i)
\\
&=& \left[(1-\varepsilon_0) \frac{x(NR[x,\varepsilon_0,p]|k)}{\bar{x}^p(NR[x,\varepsilon_0,p])} +\varepsilon_0\right] \bar{x}^p(i)
\\
&=&x'(i|k).
\end{eqnarray*}
Moreover, for all $k \in K$ and $i \in I$,
\begin{equation*}
x'(i|k) \geq \varepsilon_0 \bar{x}^p(i) \geq \varepsilon \bar{x}^p(i), 
\end{equation*} 
hence $p^{x'}(k|i) \geq \varepsilon p(k)$, 
and $x'$ is $\varepsilon$-ambiguous at $p$. 
\end{proof}
The next two propositions explain how posterior beliefs are affected by the concise action mapping: 
\begin{proposition} \label{prop_comb}
Let $\varepsilon \in (0,1]$, $x \in X$, $p \in \Delta(K)$ and $x':=c_\varepsilon(x,p)$. Then 
$$
p^{x'}(.|i) = \left\{
    \begin{array}{ll}
      \displaystyle   (1-\varepsilon) \cdot \sum_{i' \in NR[x,\varepsilon,p]} \frac{\bar{x}^p(i')}{\bar{x}^p(NR[x,\varepsilon,p])} \cdot p^{x}(.|i')+\varepsilon \cdot p & \mbox{when} \ i \in NR[x,\varepsilon,p]  \\ 
        (1-\varepsilon) \cdot p^{x}(.|i)+\varepsilon \cdot p& \mbox{otherwise.}
    \end{array}
\right.
$$
\end{proposition}
\begin{proof}
Let $i \in NR[x,\varepsilon,p]$ and $k \in K$. We have
\begin{equation*}
p^x(k|i)=\frac{x(i|k)p(k)}{\bar{x}^p(i)},
\end{equation*}
and
\begin{eqnarray*}
p^{x'}(k|i)&=&\frac{x'(i|k)p(k)}{\bar{x}^p(i)}
\\
&=&
(1-\varepsilon) \frac{x(NR[x,\varepsilon,p]|k)}{\bar{x}^p(NR[x,\varepsilon,p])}p(k)
+\varepsilon p(k)
\\ 
&=& 
(1-\varepsilon)\sum_{i' \in NR[x,\varepsilon,p]} \frac{x(i'|k)p(k)}{\bar{x}^p(NR[x,\varepsilon,p])}
+\varepsilon p(k)
\\
&=&
(1-\varepsilon) \sum_{i' \in NR[x,\varepsilon,p]}  \frac{\bar{x}^p(i')}{\bar{x}^p(NR[x,\varepsilon,p])}  p^x(k|i')
+\varepsilon p(k). 
\end{eqnarray*}
For $i \in R[x,\varepsilon,p]$, $x'(i|k)=(1-\varepsilon)x(i|k)+\varepsilon \bar{x}^p(i)$ for all $k \in K$, hence 
\\
$p^{x'}(.|i)=(1-\varepsilon) \cdot p^x(.|i)+\varepsilon \cdot p$. The last relation also holds when $\bar{x}^p(i)=\overline{x'}^p(i)=0$, since for such $i$, we have taken the convention 
$p^{x'}(.|i)=p^{x}(.|i)=p$. 
\end{proof} 
\begin{proposition} \label{concise_close}
Let $\varepsilon \in (0,1/4]$, $\varepsilon_0:=\frac{1-\sqrt{1-4\varepsilon}}{2}$, $x \in X$, $p \in \Delta(K)$ and $x':=c_{\varepsilon_0}(x,p)$. Then for all $i \in I$, 
\begin{equation*}
\left\|p^{x'}(.|i)-p^x(.|i)\right\|_1 \leq 6\varepsilon.
\end{equation*}
\end{proposition}
\begin{proof}
Let $i \in NR[x,\varepsilon_0,p]=NR[x',\varepsilon,p]$. By definition of $NR[.]$, for all $k \in K$, 
$|x'(i|k)-x(i|k)| \leq 2 \varepsilon_0 \bar{x}^p(i) \leq 6 \varepsilon \bar{x}^p(i)$, hence
\begin{eqnarray*}
\left\|p^{x'}(.|i)-p^x(.|i)\right\|_1&=& \frac{\sum_{k \in K} |x'(i|k)-x(i|k)| p(k)}{\bar{x}^p(i)}
\\
&\leq& 2 \varepsilon_0 \leq 6 \varepsilon. 
\end{eqnarray*}
%
\end{proof}
Images of mixed actions by silent mappings have hence many interesting properties, that are summarized in the following definition: 
\begin{definition}
Let $\varepsilon>0$, $p \in \Delta(K)$, and $x \in X$. We say that $x' \in X$ is an \textit{$\varepsilon$-convexification} of $x$ at $p$ if there exists $(\beta_{ii'}) \in \m{R}_+^{I \times I}$ such that all the following conditions are satisfied:
\begin{enumerate}
\item
\label{convex_combination}
\begin{equation*}
\forall i \in I, \quad p^{x'}(.|i)=\sum_{i'} \beta_{ii'} p^{x}(.|i')
\end{equation*}
\item
\label{mass_conservation}
\begin{equation*}
\forall i \in I, \quad \bar{x'}^p(i)=\bar{x}^p(i)
\end{equation*}
\item
\label{opt_transport}
\begin{equation*}
\forall i' \in I', \quad \sum_{i \in I} \beta_{ii'} \bar{x}^p(i)=\bar{x}^p(i')
\end{equation*}
\item
\label{close_belief}
\begin{equation*}
\forall i \in I, \quad
\left\|p^{x'}(.|i)-p^{x}(.|i)\right\|_1 \leq \varepsilon,
\end{equation*}
which implies in particular
\begin{equation*}
\forall \omega^e \in \Omega^e, \ \forall y \in Y, \quad \left|g^e(\omega^e,x',y)-g^e(\omega^e,x,y)\right| \leq \varepsilon \left\|g\right\|_\infty. 
\end{equation*}
\end{enumerate}
\end{definition}
The fact that Property \ref{close_belief} implies the above inequality stems from the fact that for all $(p,q,\omega) \in \Omega^e$,
\begin{eqnarray*}
\left|g^e(p,q,\omega,x',y)-g^e(p,q,\omega,x,y)\right| &\leq& \sum_{i \in I, k \in K} p(k) \left|x'(i|k)-x(i|k)\right| \left\|g\right\|_{\infty}
\\
&=& \sum_{i \in I, k \in K} \bar{x}^p(i) \left|p^{x'}(k|i)-p^{x}(k|i)\right| \left\|g\right\|_{\infty} 
\\
&=& \sum_{i \in I} \bar{x}^p(i) \left\|p^{x'}(.|i)-p^{x}(.|i)\right\|_1 \left\|g\right\|_{\infty} 
\\
&\leq& \varepsilon \left\|g\right\|_{\infty}.
\end{eqnarray*}

\begin{proposition} \label{prop_conv}
Let $\varepsilon \in (0,1/4]$, $\varepsilon_0:=\frac{1-\sqrt{1-4\varepsilon}}{2}$, $x \in X$, $p \in \Delta(K)$ and $x':=c_{\varepsilon_0}(x,p)$. Then $x'$ is a $6 \varepsilon$-convexification of $x$ at $p$. 
\end{proposition}
\begin{proof}
Define 
$$
\beta_{i i'}= \left\{
    \begin{array}{ll}
      \displaystyle   \left[(1-\varepsilon_0)  \frac{1}{\bar{x}^p(NR[x,\varepsilon_0,p])} 1_{i' \in NR[x,\varepsilon_0,p])}+\varepsilon_0\right] \bar{x}^p(i') & \mbox{when} \ i \in NR[x,\varepsilon_0,p]  \\ 
        (1-\varepsilon_0) 1_{i'=i}+\varepsilon_0 \bar{x}^p(i') & \mbox{otherwise.}
    \end{array}
\right.
$$
Proposition \ref{prop_comb} implies Property \ref{convex_combination}, Proposition \ref{prop_conservation} implies Property \ref{mass_conservation}, and Proposition \ref{concise_close} implies Property \ref{close_belief}. It remains to check Property \ref{opt_transport}. Let $i' \in I$. We have
\begin{eqnarray*}
\sum_{i \in I} \beta_{ii'} \bar{x}^p(i)&=& \sum_{i \in NR[x,\varepsilon_0,p]} \beta_{ii'} \bar{x}^p(i)+\sum_{i \in I \setminus NR[x,\varepsilon_0,p]} \beta_{ii'} \bar{x}^p(i)
\\
&=& \left[(1-\varepsilon_0)1_{i' \in NR[x,\varepsilon_0,p]}+\varepsilon_0 \bar{x}^p(NR[x,\varepsilon_0,p])\right] \bar{x}^p(i')
\\
&+&(1-\varepsilon_0) 1_{i' \in I \setminus NR[x,\varepsilon_0,p]}\bar{x}^p(i')+\varepsilon_0 \bar{x}^p(I \setminus NR[x,\varepsilon_0,p]) \bar{x}^p(i')
\\
&=&
\bar{x}^p(i'). 
\end{eqnarray*}
Hence, $x'$ is a $6\varepsilon$-convexification of $x$ at $p$.
\end{proof}
\begin{proposition} \label{opt_concise}
Assume that $\Gamma$ is an absorbing game with incomplete information on both sides. Let $\varepsilon>0$ and $\lambda \in (0,1]$. Let $\sigma$ be an optimal stationary strategy in $\Gamma^e_\lambda$, and $\sigma'$ a stationary strategy such that for all $\omega^e=(p,q,\omega) \in \Omega^e$, $\sigma'(\omega^e)$ is an $\varepsilon$-convexification of $\sigma(\omega^e)$ at $p$. Then $\sigma'$ is $2\varepsilon \left\|g\right\|_{\infty}$-optimal in $\Gamma^e_\lambda$. 
\end{proposition}
\begin{remark}
When $x$ is an optimal mixed action in the Shapley equation \eqref{eq:shapley_e} at $(p,q,\omega) \in \Omega^e$, and $x'$ is an $\varepsilon$-convexification of $x$ at $p$, for small $\varepsilon$, intuitively $x'$ should also be ``good" at $(p,q,\omega)$. 
Indeed, Property \ref{convex_combination} ensures that $x'$ discloses less information than $x$, while Property \ref{mass_conservation} implies that the law on the future state in $\Omega$ is the same under $x$ and $x'$. Moreover, Property \ref{close_belief} guarantees that stage payoffs under $x$ and $x'$ are $\varepsilon \left\|g\right\|_\infty$-close. These intuitions are correct, but they neglect an important aspect: the correlation between future state in $\Omega$ and posterior belief. Indeed, even if $x'$ is less informative than $x$ ``on average", it may be that $x'$ correlates in a bad way the future state and the posterior belief. To put things in a more formal way, consider $p,p' \in \Delta(K)$, $\beta>0$, $p:=\beta \cdot p'+(1-\beta) \cdot p''$, and $\omega,\omega' \in \Omega$. Even though $v_\lambda$ is concave in $p$, it may be that
\begin{equation*}
\beta v_\lambda(p',q,\omega)+(1-\beta)v_\lambda(p'',q,\omega') >\beta v_\lambda(p,q,\omega)
+(1-\beta) v_\lambda(p,q,\omega')
\end{equation*}
Hence, proving that $\sigma'$ is indeed $O(\varepsilon)$-optimal is trickier than a mere concave inequality applied ``stage by stage", and requires Property \ref{opt_transport}, as well as the fact that the game is absorbing. 
\end{remark}

We need first the following lemma: 
\begin{lemma} \label{lemma:future_payoff}
Assume that $\Gamma$ is an absorbing game with incomplete information on both sides, and let $\varepsilon>0$.
Let $u,v:\Omega^e \rightarrow \m{R}$ two functions satisfying the two following properties: 
\begin{enumerate}[(i)]
\item \label{concise_property_1}
For all $\omega \neq \omega^0$ and $(p',q') \in \Delta(K) \times \Delta(L)$, $u(p',q',\omega)=v(p',q',\omega)$
\item \label{concise_property_2}
There exists $C>0$ such that for all $q' \in \Delta(L)$, $v(.,q',\omega^0): (\Delta(K),\left\|.\right\|_1) \rightarrow (\m{R},|.|)$ is concave and $C$-Lipschitz.
\end{enumerate}
Let $(p,q) \in \Delta(K) \times \Delta(L)$, $x \in X$, $x'$ an $\varepsilon$-convexification of $x$ at $p$, and $y \in Y$. 
For all $j \in J$, let $r(j):=\sum_{i \in I} \rho(\omega^0|\omega^0,i,j) \bar{x}^p(i)$, and $R:=\sum_{j \in J} \bar{y}^q(j) r(j)$. Then 
\begin{equation*}
\m{E}^e_{p,q,\omega^0,x',y}(u)-\m{E}^e_{p,q,\omega^0,x,y}(v) \geq \min_{(p',q') \in \Delta(K) \times \Delta(L)} \left\{u(p',q',\omega^0)-v(p',q',\omega^0)\right\} R -2(1-R) \varepsilon C.
\end{equation*}
\end{lemma}
\begin{proof}
\begin{eqnarray*}
&&\m{E}^e_{p,q,\omega^0,x',y}(u)-\m{E}^e_{p,q,\omega^0,x,y}(v)
\\
&=& \sum_{i \in I} \sum_{j \in J} \sum_{\omega \in \Omega} \rho(\omega|\omega^0,i,j) 
\bar{x}^p(i) \bar{y}^q(j) \cdot [u(p^{x'}(.|i),q^y(.|j), \omega)-v( p^{x}(.|i),q^y(.|j), \omega)]
\\
&=&
\sum_{j \in J} \bar{y}^q(j) \underbrace{\sum_{i \in I} \sum_{\omega \in \Omega} \rho(\omega|\omega^0,i,j) 
\bar{x}^p(i)  \cdot [u(p^{x'}(.|i),q^y(.|j), \omega)-v( p^{x'}(.|i),q^y(.|j), \omega)]}_{A(j)} 
\\
&+& \sum_{j \in J} \bar{y}^q(j) \underbrace{\sum_{i \in I} \sum_{\omega \in \Omega} \ \rho(\omega|\omega^0,i,j) 
\bar{x}^p(i) \cdot [v(p^{x'}(.|i),q^y(.|j), \omega)-v( p^{x}(.|i),q^y(.|j), \omega)]}_{B(j)}
\end{eqnarray*}
Let $j \in J$. Let us first take care of the term $A(j)$. Let $i \in I$. For all $\omega \neq \omega^0$, by Property \ref{concise_property_1}, we have $u(p^{x'}(.|i),q^y(.|j), \omega)=v(p^{x'}(.|i),q^y(.|j), \omega)$. 
\\
Moreover, setting $D:=\min_{(p',q') \in \Delta(K) \times \Delta(L)} \left\{u(p',q',\omega^0)-v(p',q',\omega^0)\right\}$, we have 
\\
$u(p^{x'}(.|i),q^y(.|j), \omega^0)-v(p^{x'}(.|i),q^y(.|j), \omega^0) \geq D$. We deduce that
\begin{eqnarray}
A(j) &=&
\sum_{i \in I} \rho(\omega^0|\omega^0,i,j) 
\bar{x}^p(i) \cdot [u(p^{x'}(.|i),q^y(.|j), \omega^0)-v( p^{x'}(.|i),q^y(.|j), \omega^0)]
\nonumber \\
&\geq& 
\left(\sum_{i \in I}  \rho(\omega^0|\omega^0,i,j) 
\bar{x}^p(i)  \right) \cdot D
\nonumber \\
&=& r(j)D.  \label{ineq_Aj}
\end{eqnarray}
Consider now the term $B(j)$. We have
\begin{eqnarray*}
B(j)&=&
\sum_{i \in I}\sum_{\omega \in \Omega \setminus \left\{\omega^0\right\}} \ \rho(\omega|\omega^0,i,j) 
\bar{x}^p(i) \cdot [v(p^{x'}(.|i),q^y(.|j), \omega)-v(p^{x}(.|i),q^y(.|j), \omega)]
\\
&+&
\sum_{i \in I} \left[1-\sum_{\omega \in \Omega \setminus \left\{\omega^0\right\}} \ \rho(\omega|\omega^0,i,j))\right]
\bar{x}^p(i) \cdot [v(p^{x'}(.|i),q^y(.|j), \omega^0)-v( p^{x}(.|i),q^y(.|j), \omega^0)]
\end{eqnarray*}
By \ref{concise_property_2} and Property \ref{close_belief} of $\varepsilon$-convexification, for all $\omega \in \Omega$, we have 
\\
$|v(p^{x'}(.|i),q^y(.|j), \omega)-v( p^{x}(.|i),q^y(.|j), \omega)| \leq \varepsilon C$.
Hence,
\begin{eqnarray*}
B(j)&\geq&
\sum_{i \in I}\sum_{\omega \in \Omega \setminus \left\{\omega^0\right\}} \ \rho(\omega|\omega^0,i,j) 
\bar{x}^p(i) \cdot [-\varepsilon C]
\\
&+&
\sum_{i \in I}
\bar{x}^p(i) \cdot [v(p^{x'}(.|i),q^y(.|j), \omega^0)-v( p^{x}(.|i),q^y(.|j), \omega^0)]
\\
&-&  \sum_{i \in I}\sum_{\omega \in \Omega \setminus \left\{\omega^0\right\}} \ \rho(\omega|\omega^0,i,j) 
\bar{x}^p(i) \cdot [\varepsilon C]
\\
&=& \sum_{i \in I}
\bar{x}^p(i) \cdot [v(p^{x'}(.|i),q^y(.|j), \omega^0)-v( p^{x}(.|i),q^y(.|j), \omega^0)]-2  (1-r(j)) \varepsilon C 
\end{eqnarray*}
Moreover,  
\begin{align*}
\sum_{i \in I} 
\bar{x}^p(i) \cdot v(p^{x'}(.|i),q^y(.|j), \omega^0) & \geq  
\sum_{i \in I} 
\bar{x}^p(i)  \cdot \sum_{i' \in I} \beta_{ii'}  v(p^x(.|i'),q^y(.|j), \omega^0) \\ 
& \text{Property \ref{convex_combination} and concavity of $v$ in the first variable}
\\
&= 
\sum_{i' \in I} 
  \left[ \sum_{i \in I} \beta_{ii'} \bar{x}^p(i)\right]  v(p^x(.|i'),q^y(.|j), \omega^0) &
\\
&= 
\sum_{i' \in I} \bar{x}^p(i') v(p^x(.|i'),q^y(.|j), \omega^0) \quad   ; \text{Property \ref{opt_transport}}
\end{align*}
It follows that
\begin{equation*}
B(j) \geq -2(1-r(j)) \varepsilon C,
\end{equation*}
and combining with \eqref{ineq_Aj},
\begin{eqnarray*}
\m{E}^e_{p,q,\omega^0,x',y}(u)-\m{E}^e_{p,q,\omega^0,x,y}(v) &\geq& 
\sum_{j \in J} \bar{y}^q(j) [r(j)D-2(1-r(j))\varepsilon C]
\\
&\geq&
D R -2(1-R) \varepsilon C.
\end{eqnarray*}
\end{proof}
\begin{proof}[Proof of Proposition \ref{opt_concise}]
Let $\tau: \Omega^e \rightarrow Y$ be a stationary strategy for Player 2. 
Let $p,q$ that minimize $(p,q) \rightarrow \gamma^e_\lambda(p,q,\omega^0,\sigma',\tau)-v^e_\lambda(p,q,\omega^0)$, and  call this value $D$. It is enough to prove that 
$D \geq -2\varepsilon \left\|g\right\|_{\infty}$.
\\
Let $x:=\sigma(p,q,\omega)$, $x':=\sigma'(p,q,\omega)$ and $y:=\tau(p,q,\omega)$. 
By a recursive argument, we have 
\begin{equation} \label{dyn_succinct}
\gamma^e_{\lambda}(p,q,\omega^0,\sigma',\tau)=\lambda g^e(p,q,\omega^0,x',y)+(1-\lambda) \m{E}^e_{p,q,\omega^0,x',y}(\gamma^e_{\lambda}(\, . \, ,\sigma',\tau)),
\end{equation}
and by Shapley equation (see Proposition \ref{prop_recursive} \ref{prop:shapley_e}),
\begin{equation} \label{dyn_opt}
v^e_{\lambda}(p,q,\omega^0) \leq \lambda g^e(p,q,\omega^0,x,y)+(1-\lambda) \m{E}^e_{p,q,\omega^0,x,y}(v^e_{\lambda}).
\end{equation}
By Property \ref{close_belief}, we have
\begin{equation} \label{payoff_succinct}
g^e(p,q,\omega^0,x',y) \geq  g^e(p,q,\omega^0,x,y)- \varepsilon \left\|g\right\|_{\infty}. 
\end{equation}
Let us focus on the second term. 
 Applying the previous lemma with $u=\gamma_\lambda(.,\sigma',\tau)$, $v=v^e_\lambda$ and $C=\left\|g\right\|_\infty$, and using the same notations for $R$, we get
\begin{equation} \label{ineq:xx'}
\m{E}^e_{p,q,\omega^0,x',y}(\gamma^e_{\lambda})-\m{E}^e_{p,q,\omega^0,x,y}(v^e_{\lambda}) \geq DR -2(1-R) \varepsilon \left\|g\right\|_{\infty}
\end{equation}
Combining (\ref{dyn_succinct}), \eqref{dyn_opt} and (\ref{payoff_succinct}), we get
\begin{eqnarray*}
\gamma^e_{\lambda}(p,q,\omega^0,\sigma',\tau)-v^e_\lambda(p,q,\omega^0) &\geq& -\lambda  \varepsilon \left\|g\right\|_\infty+(1-\lambda)(DR -2(1-R) \varepsilon \left\|g\right\|_{\infty})
\\
&=& -\lambda  \varepsilon \left\|g\right\|_\infty+(1-\lambda)(D-(1-R)(D+2\varepsilon \left\|g\right\|_{\infty})),
\end{eqnarray*}
which implies by definition of $D$
\begin{equation*}
D \geq -\lambda  \varepsilon \left\|g\right\|_\infty+(1-\lambda)(D-(1-R)(D+2\varepsilon \left\|g\right\|_{\infty})). 
\end{equation*}
Recall that our aim is to prove that $D \geq -2\varepsilon \left\|g\right\|_{\infty}$. Assume by contradiction that $D< -2\varepsilon \left\|g\right\|_{\infty}$. It follows that
\begin{equation*} 
D \geq -\lambda \varepsilon \left\|g\right\|_\infty+(1-\lambda)D,
\end{equation*}
and thus
\begin{equation*}
D \geq -\varepsilon \left\|g \right\|_\infty \geq -2\varepsilon \left\|g \right\|_\infty, 
\end{equation*}
which is a contradiction. 
\end{proof}
We are now ready to prove Proposition \ref{prop:opt_concise}. 
\begin{proof}[Proof of Proposition \ref{prop:opt_concise}]
Let $\Gamma=(K,L,\Omega,I,J,\rho,g)$ be an absorbing game with incomplete information on both sides. Let $\varepsilon \in (0,1/4]$, $\lambda \in (0,1]$, and $\sigma$ be an optimal stationary strategy for Player 1 in $\Gamma^e_\lambda$. Let $\varepsilon_0:=\frac{1-\sqrt{1-4\varepsilon}}{2}$.
Define a stationary strategy $\sigma'$ by: for all $(p,q,\omega) \in \Omega^e$, $\sigma'(p,q,\omega):=c_{\varepsilon_0}(\sigma(p,q,\omega),p)$. By Proposition \ref{prop_concise_mapping}, $\sigma'$ is $\varepsilon$-concise and $\varepsilon$-ambiguous. According to Proposition \ref{prop_conv}, for all $(p,q,\omega) \in \Omega^e$, 
$\sigma'(p,q,\omega)$ is a $6\varepsilon$-convexification of $\sigma(p,q,\omega)$ at $p$. By Proposition \ref{opt_concise}, $\sigma'$ is $12\varepsilon \left\|g\right\|_{\infty}$-optimal in 
$\Gamma^e_\lambda$, and the result follows. 
\end{proof}
\subsection{Martingale of beliefs under concise strategies}
This section finishes on the proof of Proposition \ref{prop:bounded_var} by giving an upper bound on the $L^1$-variation of the belief process generated by an $\varepsilon$-concise strategy. Results do not require the absorbing assumption, hence a stochastic game with incomplete information on both sides $\Gamma=(K,L,\Omega,I,J,\rho,g)$ is fixed for the remainder of this section. The goal is to prove the following:
\begin{proposition} \label{variation_1_e}
Let $\varepsilon \in (0,1]$, $\omega^e \in \Omega^e$, $\sigma$ an $\varepsilon$-concise strategy and $\tau$ a strategy of Player 2 in $\Gamma^e$. 
Let $T:=\max\left\{m \geq 1, p_m \in \Delta(K) \setminus F_{\varepsilon} \right\}$. Then
\begin{equation*}
\m{E}^{e}_{\omega^e,\sigma,\tau}\left(\sum_{m=1}^{T} \left\|p_{m+1}-p_m\right\|_1\right) \leq 3 \sqrt{|K|} \varepsilon^{-5}. 
\end{equation*}
\end{proposition}

The two following inequalities play a crucial role.
\begin{lemma} \label{jump}
Let $\varepsilon>0$ and $p \in \Delta(K)$ such that for all $k \in K$, $p(k)>0$. Let $x \in X$ be an $\varepsilon$-concise mixed action at $p$. Then 
\begin{enumerate}[(i)]
\item \label{jump_nr}
For all $i \in NR[x,\varepsilon,p]$, 
\begin{equation*}
\left\|p^{x}(.|i)-p \right\|_1 \leq 2 \left\|\frac{1}{p}\right\|_{\infty} \frac{\bar{x}^p(R[x,\varepsilon,p])}{\bar{x}^p(NR[x,\varepsilon,p])}
\end{equation*}
\item \label{jump_r}
For all $i \in R[x,\varepsilon,p]$,
\begin{equation*}
\left\|p^{x}(.|i)-p \right\|_1 \geq \varepsilon \min_{k \in K} p(k).
\end{equation*}
\end{enumerate}
\end{lemma}
\begin{proof}
\begin{enumerate}[(i)]
\item
Let $i \in NR[x,\varepsilon,p]$ and $k \in K$. We have
\begin{eqnarray*}
\left|p^{x}(k|i)-p(k)\right|&=&\left| \frac{x(NR[x,\varepsilon,p]|k)p(k)}{\bar{x}^p(NR[x,\varepsilon,p])}-p(k) \right|
\\
&=& \left| \frac{x(NR[x,\varepsilon,p]|k)-\bar{x}^p(NR[x,\varepsilon,p])}{\bar{x}^p(NR[x,\varepsilon,p])} \right| p(k)
\\
&=& \left| \frac{x(R[x,\varepsilon,p]|k)-\bar{x}^p(R[x,\varepsilon,p])}{\bar{x}^p(NR[x,\varepsilon,p])} \right| p(k)
\\
& \leq & \left(\frac{x(R[x,\varepsilon,p]|k)}{\bar{x}^p(NR[x,\varepsilon,p])}+\frac{\bar{x}^p(R[x,\varepsilon,p])}{\bar{x}^p(NR[x,\varepsilon,p])}\right) p(k)
\end{eqnarray*}
Moreover, we have 
\begin{eqnarray*}
\bar{x}^p(R[x,\varepsilon,p])=\sum_{k' \in K} p(k') x(R[x,\varepsilon,p]|k') \geq p(k) x(R[x,\varepsilon,p]|k),
\end{eqnarray*}
hence 
\begin{equation*}
\frac{x(R[x,\varepsilon,p]|k)}{\bar{x}^p(R[x,\varepsilon,p])} \leq \left\|\frac{1}{p}\right\|_{\infty},
\end{equation*}
and thus
\begin{equation*}
\frac{x(R[x,\varepsilon,p]|k)}{\bar{x}^p(NR[x,\varepsilon,p])} \leq \left\|\frac{1}{p}\right\|_{\infty} \frac{\bar{x}^p(R[x,\varepsilon,p])}{\bar{x}^p(NR[x,\varepsilon,p])}.
\end{equation*}
 We deduce that
\begin{eqnarray*}
\left|p^{x}(k|i)-p(k)\right| \leq \left(1+\left\|\frac{1}{p}\right\|_{\infty}\right) \frac{\bar{x}^p(R[x,\varepsilon,p])}{\bar{x}^p(NR[x,\varepsilon,p])}   p(k),
\end{eqnarray*}
and the result follows. 
\item
Let $i \in R[x,\varepsilon,p]$. By definition, there exists $k \in K$ such that either $x(i|k) < (1-\varepsilon)\bar{x}^p(i)$, or $x(i|k) > (1+\varepsilon)\bar{x}^p(i)$. It follows that either 
$p^x(k|i) < (1-\varepsilon) p(k)$, or $p^x(k|i) > (1+\varepsilon) p(k)$, and the result follows. 
\end{enumerate}
\end{proof}
Properties \ref{jump_nr} and \ref{jump_r} imply that while $(p_m)$ stays at a distance larger than $\varepsilon$ from the frontier of $\Delta(K)$, either $\left\|p_{m+1}-p_m\right\|_1$ is larger than $\varepsilon^2$, or 
the order of magnitude of $\left\|p_{m+1}-p_m\right\|_1$ is smaller than the probability that $\left\|p_{m+1}-p_m\right\|$ is larger than $\varepsilon^2$. Since a martingale can not make too many significant ``jumps", this implies an upper bound on the expectation of $\sum_{m \geq 1} \left\|p_{m+1}-p_m\right\|_1$. This argument is made formal in what follows. 
\begin{proof}[Proof of Proposition \ref{variation_1_e}]
For $m \geq 1$, denote $h_m:=(\omega^e_1,x_1,y_1,\dots,\omega^e_{m-1},x_{m-1},y_{m-1},\omega^e_m)$ the random history at stage $m$. Let $R_m:=R[x_m,\varepsilon,p_m]$, and $NR_m:=NR[x_m,\varepsilon,p_m]$.
For each $m \geq 1$,
\begin{eqnarray}
1_{m \leq T}\m{E}^e_{\omega^e,\sigma,\tau}(\left\|p_{m+1}-p_m\right\|_1  | h_m) &=&
1_{m \leq T}\sum_{i \in NR_m} \overline{x_m}^{p_m}(i) \left\|p_m^{x_m}(.|i)-p_m\right\|_1 \nonumber
\\
&+&
1_{m \leq T}\sum_{i \in R_m} \overline{x_m}^{p_m}(i) \left\|p_m^{x_m}(.|i)-p_m\right\|_1 \label{eq_split}
\end{eqnarray}
We start by bounding the first right-hand side term. Let $i \in NR_m$. We have
\begin{align*}
1_{m \leq T} \sum_{i \in NR_m} \overline{x_m}^{p_m}(i) \left\|p_m^{x_m}(.|i)-p_m\right\|_1 &\leq
1_{m \leq T} \sum_{i \in NR_m} \overline{x_m}^{p_m}(i) 2\left\|\frac{1}{p_m}\right\|_{\infty} \frac{\overline{x_m}^{p_m}(R_m)}{\overline{x_m}^{p_m}(NR_m)} & ; \text{Lemma \ref{jump} \ref{jump_nr}} 
\\
&=
1_{m \leq T}  2\left\|\frac{1}{p_m}\right\|_{\infty} \overline{x_m}^{p_m}(R_m)  &
 \\ 
 &\leq 1_{m \leq T} 2 \varepsilon^{-1} \overline{x_m}^{p_m}(R_m).  & ; p_m \notin F_{\varepsilon}\text{ on }\left\{m \leq T \right\} \numberthis \label{ineq_mart_1}
 \end{align*}
Moreover, we have 
\begin{align}
1_{m \leq T} \overline{x_m}^{p_m}(R_m)
&= 1_{m \leq T} \m{P}^{e}_{\omega^e,\sigma,\tau}(\exists i \in R_m, \ p_{m+1}=p_m^{x_m}(.|i) | h_m)  \nonumber
\\
&\leq 1_{m \leq T} \m{P}^e_{\omega^e,\sigma,\tau}(\left\|p_{m+1}-p_m\right\|_1 \geq \varepsilon \min_{k \in K} p_m(k) | h_m)  & \text{Lemma \ref{jump} \ref{jump_r}} \nonumber
\\
&\leq 1_{m \leq T} \m{P}^e_{\omega^e,\sigma,\tau}(\left\|p_{m+1}-p_m\right\|_1 \geq \varepsilon^2 | h_m) & p_m \notin F_{\varepsilon} \nonumber
\\
&= 1_{m \leq T} \m{P}^e_{\omega^e,\sigma,\tau}(\left\|p_{m+1}-p_m\right\|_1^2 \geq \varepsilon^4 | h_m)  & \nonumber
\\
&\leq 1_{m \leq T} \varepsilon^{-4} \m{E}^e_{\omega^e,\sigma,\tau}(\left\|p_{m+1}-p_m\right\|_1^2 | h_m) & \text{Markov inequality} \label{ineq_mart_2}
\end{align}
Combining (\ref{ineq_mart_1}) and (\ref{ineq_mart_2}) yield
\begin{equation}
1_{m \leq T} \sum_{i \in NR_m} \overline{x_m}^{p_m}(i) \left\|p_m^{x_m}(.|i)-p_m\right\|_1
\leq 1_{m \leq T} 2\varepsilon^{-5} \m{E}^e_{\omega^e,\sigma,\tau}(\left\|p_{m+1}-p_m\right\|_1^2 | h_m). \label{ineq_mart_3}
\end{equation}
We now deal with the second right-hand side term in \eqref{eq_split}. 
Let $i \in R_m$. We have 
\begin{align*}
1_{m \leq T} \left\|p_m^{x_m}(.|i)-p_m\right\|_1
&= 1_{m \leq T} \left\|p_m^{x_m}(.|i)-p_m\right\|_1^2/\left\|p_m^{x_m}(.|i)-p_m\right\|_1 
\\ 
&\leq 1_{m \leq T} \left\|p_m^{x_m}(.|i)-p_m\right\|_1^2/\varepsilon^2  \quad \text{Lemma \ref{jump} \ref{jump_r} and} \ p_m \notin F_{\varepsilon} \ \text{on} \ \left\{m \leq T\right\}
\end{align*}
We deduce that 
\begin{eqnarray*} 
1_{m \leq T}\sum_{i \in R_m} \overline{x_m}^{p_m}(i) \left\|p_m^{x_m}(.|i)-p_m\right\|_1
&\leq&  1_{m \leq T} \sum_{i \in R_m} \overline{x_m}^{p_m}(i) \left\|p_m^{x_m}(.|i)-p_m\right\|_1^2/\varepsilon^2
\\
&\leq & 1_{m \leq T} \sum_{i \in I} \overline{x_m}^{p_m}(i) \left\|p_m^{x_m}(.|i)-p_m\right\|_1^2/\varepsilon^2
\\
&=& 1_{m \leq T} \varepsilon^{-2} \m{E}^e_{\omega^e,\sigma,\tau}(\left\|p_{m+1}-p_m\right\|_1^2  | h_m)
\end{eqnarray*}
Combining with \eqref{eq_split} and \eqref{ineq_mart_3}, we get 
\begin{eqnarray*}
1_{m \leq T} \m{E}^{e}_{\omega^e,\sigma,\tau}(\left\|p_{m+1}-p_m\right\|_1 | h_m)
&\leq& 1_{m \leq T} (2 \varepsilon^{-5}+\varepsilon^{-2}) \m{E}^{e}_{\omega^e,\sigma,\tau}(\left\|p_{m+1}-p_m\right\|_1^2  | h_m),
\\
&\leq & 1_{m \leq T} 3\varepsilon^{-5} \m{E}^{e}_{\omega^e,\sigma,\tau}(\left\|p_{m+1}-p_m\right\|_1^2  | h_m),
\end{eqnarray*}
hence
\begin{eqnarray*}
\m{E}^{e}_{\omega^e,\sigma,\tau}\left(\sum_{m=1}^T \left\|p_{m+1}-p_m\right\|_1\right) 
&=& \sum_{m=1}^{+\infty} \m{E}^{e}_{\omega^e,\sigma,\tau}(1_{m \leq T} \m{E}^{e}_{\omega^e,\sigma,\tau}(\left\|p_{m+1}-p_m\right\|_1  | h_m))
\\
&\leq&
3 \varepsilon^{-5} \sum_{m=1}^{+\infty} \m{E}^{e}_{\omega^e,\sigma,\tau}(1_{m \leq T} \m{E}^{e}_{\omega^e,\sigma,\tau}(\left\|p_{m+1}-p_m\right\|_1^2  | h_m))
\\
&=&
3 \varepsilon^{-5} 
\m{E}^{e}_{\omega^e,\sigma,\tau} \left(\sum_{m=1}^T \left\|p_{m+1}-p_m\right\|_1^2\right).
\end{eqnarray*}
Moreover, we have
\begin{align*}
\m{E}^{e}_{\omega^e,\sigma,\tau} \left(\sum_{m=1}^T \left\|p_{m+1}-p_m\right\|_1^2\right) &\leq \m{E}^{e}_{\omega^e,\sigma,\tau} \left(\sum_{m=1}^{+\infty} \left\|p_{m+1}-p_m\right\|_1^2\right)  & 
\\
&\leq \sqrt{|K|} \m{E}^{e}_{\omega^e,\sigma,\tau}\left(\sum_{m=1}^{+\infty} \left\|p_{m+1}-p_m\right\|_2^2\right) & \left\|.\right\|_1 \leq \sqrt{|K|} \left\|.\right\|_2 
\\ 
&\leq \sqrt{|K|} & \text{martingale property} \\ & &\text{(see e.g. \cite[Lemma 3.4, p.31]{sorin02b})},
\end{align*}
and the result follows. 
\end{proof}
The goal of this section now reached:
\begin{proof}[Proof of Proposition \ref{prop:bounded_var}]
Applying successively Proposition \ref{prop:opt_concise} and Proposition \ref{variation_1_e} readily implies Proposition \ref{prop:bounded_var}.
\end{proof}
\section{Coupling between true and approximated belief dynamics and proof of Proposition \ref{comparison_e_f_w}} \label{sec:coupling}
This section finishes on the proof of Proposition \ref{comparison_e_f_w}, which implies Theorem \ref{main_thm}, as explained in Section \ref{sec:model}.
\subsection{Modified game} \label{subsec:modif}
Let $\Gamma=(K,L,\Omega,I,J,\rho,g)$ be a stochastic game with incomplete information on both sides, and $(P,Q)$ be a triangulation of $\Delta(K) \times \Delta(L)$. To prove Proposition \ref{comparison_e_f_w}, 
we aim at making a coupling between histories in $\Gamma^e$ and histories in $\Gamma^f$, in such a way that $(1)$ the two belief dynamics over Player 1's type are close to each other, and $(2)$ the two belief dynamics over Player 2's type are identical. For $(2)$, one difficulty is that in $\Gamma^f$, the belief on Player 2's type splits over the elements of the triangulation. Such a splitting may not be generable by a mixed action in $\Gamma^e$, hence we need to enrich slightly the set of actions of Player 2 in $\Gamma^e$. This produces a game $\Gamma^{\eta}$, that is defined by the following elements:
\begin{itemize}
\item
State space $\Omega^e$,
\item
Action set $X$ for Player 1,
\item
Action set $Y^{\eta}:=Y \times \left\{F,E\right\}$ for Player 2,
\item
Transition function $\rho^{\eta}$, such that $\forall (p,q,\omega), (p',q',\omega')  \in \Omega^{e}, \ \forall (x,y) \in X \times Y^\eta$, 
\begin{eqnarray*}
\rho^{\eta}(p,q,\omega,x,(y,E))&:=&\rho^e(p,q,\omega,x,y)
\\
\rho^{\eta}(p',q',\omega'|p,q,\omega,x,(y,F))&:=&\sum_{q'' \in \Delta(L)} S[q'|q''] \rho^e(p',q'',\omega'|p,q,\omega,x,y)
\end{eqnarray*}
\item
Payoff function $g^{\eta}$ such that for all $(p,q,\omega) \in \Omega^{e}$, for all $(x,y) \in X \times Y^\eta$,
\begin{equation*}
g^{\eta}(p,q,\omega,x,(y,E))=g^{\eta}(p,q,\omega,x,(y,F)):=g^e(p,q,\omega,x,y). 
\end{equation*}
\end{itemize}
Hence, when Player 2 chooses $(y,E)$, the transition is identical to the one produced by $y$ in $\Gamma^e$. When Player 2 chooses $(y,F)$, first $(p',q'',\omega')$ is drawn with probability $\rho^e(p',q'',\omega'|p,q,\omega,x,y)$, and then $q''$ is splitted over $Q$, in the same way as in $\Gamma^f$. 

Strategies in $\Gamma^\eta$ are defined in an analogous way as in $\Gamma^e$, and are denoted respectively by $\Sigma^\eta$ and $T^\eta$ for Player 1 and 2. 
A tuple $(\omega^e,\sigma,\tau) \in \Omega^e \times \Sigma^\eta \times T^\eta$ induces a probability measure $\m{P}^\eta_{\omega^e,\sigma,\tau}$ on the set of infinite histories of the game
$(\Omega^e \times X \times Y^\eta)^{\m{N}}$, and the expectation with respect to this probability measure is denoted by $\m{E}^\eta_{\omega^e,\sigma,\tau}$. As usual, one can define the discounted game $\Gamma^\eta_\lambda$, and its value is denoted by $v^\eta_\lambda$. 

Such a modification of $\Gamma^e$ does not modify discounted values. The intuition is that playing $(y,F)$ instead of $(y,E)$ provides Player 1 with more information, hence enlarging the set of strategies of Player 2 from $Y$ to $Y^\eta$ does not bring any advantage to Player 2. 
\begin{proposition} \label{equal_e_eta}
For all $\lambda \in (0,1]$, $v^{\eta}_\lambda=v^{e}_\lambda$. Moreover, each player has a pure optimal stationary strategy. 
\end{proposition}
The proof is standard and is postponed to Subsection \ref{subsec:missing}. 

Hence, when proving Proposition \ref{comparison_e_f_w}, we will work with the game $\Gamma^\eta$ instead of the game $\Gamma^e$. 
Let us explain now why the main result of Section \ref{sec:concise}, namely Proposition \ref{prop:bounded_var}, holds in $\Gamma^{\eta}$ too. This proposition stems from Propositions \ref{prop:opt_concise} and \ref{variation_1_e}, hence this is enough to check that they hold in $\Gamma^\eta$. 

Definitions of $\varepsilon$-concise strategy, $\varepsilon$-silent mapping and $\varepsilon$-convexification do not depend on Player 2's strategy set, hence can be kept as they are, and Propositions \ref{prop_conservation}, \ref{prop_concise_mapping}, \ref{prop_comb}, \ref{concise_close}, and \ref{prop_conv} apply. Proposition \ref{opt_concise} holds too, with a slight modification of the proof. For completeness, this is done in Subsection \ref{subsec:missing}. The proof of Proposition \ref{prop:opt_concise} consists in applying successively Proposition \ref{prop_concise_mapping}, Proposition \ref{prop_conv} and Proposition \ref{opt_concise}. Hence, Proposition \ref{prop:opt_concise} also holds when replacing $\Gamma^e_\lambda$ by $\Gamma^{\eta}_\lambda$. 

As far as Proposition \ref{variation_1_e} is concerned, its proof only uses properties of the process $(p_m)$ under an $\varepsilon$-concise strategy, hence can be reproduced word for word in the game $\Gamma^{\eta}$. Finally, we obtain the equivalent of Proposition \ref{prop:bounded_var}:
\begin{proposition} \label{prop:bounded_var_eta}
Let $\Gamma=(K,L,\Omega,I,J,\rho,g)$ be an absorbing game with incomplete information on both sides.
Let $\varepsilon \in (0,1/4]$, $T:=\max\left\{m \geq 1, p_m \in \Delta(K) \setminus F_{\varepsilon} \right\}$, and $\lambda \in (0,1]$. Then Player 1 has a (pure) $12\varepsilon \left\|g\right\|_\infty$-optimal stationary strategy $\sigma$ in $\Gamma^{\eta}_\lambda$ that is $\varepsilon$-ambiguous, and such that for all $\tau \in T^{\eta}$,
\begin{equation*}
\m{E}^{\eta}_{\omega^0,\sigma,\tau}\left(\sum_{m=1}^{T} \left\|p_{m+1}-p_m\right\|_1\right) \leq 3 \sqrt{|K|} \varepsilon^{-5}. 
\end{equation*}
\end{proposition}
\begin{remark}
Instead of defining the game $\Gamma^e$ in Section \ref{sec:model} and then the game $\Gamma^{\eta}$ in this section, we could have started right from the beginning with the game $\Gamma^{\eta}$, without considering $\Gamma^e$. Nonetheless, this would have made Sections \ref{sec:auxiliary} and \ref{sec:concise} heavier and less readable. Since the game $\Gamma^{\eta}$ is only useful at the end of the proof of Proposition \ref{comparison_e_f_w}, it is preferable to introduce it only in this section.  
\end{remark}
Recall that for $\varepsilon>0$, $F_\varepsilon:=\left\{p \in \Delta(K) \ | \exists k \in K, \  p(k) \leq \varepsilon \right\}$. In addition, the following technical result is needed:
\begin{proposition} \label{lipf}
Let $\varepsilon>0$ and $\Gamma=(K,L,\Omega,I,J,\rho,g)$ be a stochastic game with incomplete information on both sides that satisfies Assumption \ref{asslip}. Let $\alpha>0$ and $(P,Q)$ be a triangulation of $\Delta(K) \times \Delta(L)$ with stepsize smaller than $\alpha$. Let $\lambda \in (0,1]$, $p \in F_{\varepsilon} \cap P$, and $(q,\omega) \in Q \times \Omega$. Then
\begin{equation*}
v^f_{\lambda}(p,q,\omega) \geq v_\lambda^{\eta}(p,q,\omega)-\left\|v^f_{\lambda}-v^{\eta}_{\lambda}\right\|_{F_0}-(4 \varepsilon+\alpha) \left\|g\right\|_{\infty},
\end{equation*}
where $\left\|v^f_{\lambda}-v^{\eta}_{\lambda}\right\|_{F_0}:=\sup_{(p,q,\omega) \in (F_0 \cap P) \times Q \times \Omega} \left|v^f_{\lambda}(p,q,\omega)-v^{\eta}_{\lambda}(p,q,\omega)\right|$.
\end{proposition}
\begin{remark}
This result is the reason why we introduced Assumption \ref{asslip}, and the author does not know whether it holds when dropping that assumption. Fortunately, recall that Assumption \ref{asslip} is w.l.o.g. for the proof of Theorem \ref{main_thm}. Note that if $v^f_\lambda$ is $\left\|g\right\|_\infty$-Lipschitz with respect to its first variable, the result is trivial since $v^\eta_\lambda$ is always $\left\|g\right\|_\infty$-Lipschitz.
\end{remark}
The proof is also postponed to Subsection \ref{subsec:missing}. 
\subsection{Translation mapping}
Fix a stochastic game with incomplete information on both sides $\Gamma=(K,L,\Omega,I,J,\rho,g)$. 
This subsection introduces an important tool that will be used in the coupling between $\Gamma^{\eta}$ and $\Gamma^f$. 
Denote $X_0:=\left\{(x,p,p') \in X \times \Delta(K)^2, \forall (i,k) \in I \times K, \ {p}'(k)+p^x(k | i)-p(k) \geq 0 \ \text{and} \ p'(k)>0 \right\}$. 
\begin{definition} \label{def_translation}
Define the \textit{translation mapping} $T: X \times \Delta(K)^2 \rightarrow X$ by: for all $(k,i) \in K \times I$,
\\
$$
[T(x,p,p')](i |k ) = \left\{
    \begin{array}{ll}
      \frac{\bar{x}^p(i)}{p'(k)} \left[{p}'(k)+p^x(k | i)-p(k)\right]
 & \mbox{when} \ (x,p,p') \in X_0
 \\
      x(i|k)  & \mbox{otherwise}.
    \end{array}
\right.
$$
\end{definition}
Note that the definition of $T(x,p,p')$ in the second case is simply a convention that will be convenient later on. 

\begin{proposition} \label{prop_translation}
Let $(x,p,p') \in X \times \Delta(K)^2$ and $x':=T(x,p,p')$.  Then the following assertions hold:
\begin{enumerate}[(i)]
\item \label{action_conservation}
\begin{equation*}
\forall i \in I, \quad \overline{x'}^{p'}(i)=\bar{x}^p(i)
\end{equation*}
\item \label{error_conservation}
Assume $(x,p,p') \in X_0$. Then
\begin{equation*}
\forall i \in I \ \forall k \in K, \quad {p'}^{x'}(k|i)-p^{x}(k|i)=p'(k)-p(k).
\end{equation*}
\item
Assume that for all $k \in K$, $p(k)>0$. Then 
\begin{equation*}
\forall i \in I, \ \forall k \in K, \quad |x'(i|k) - x(i|k)| \leq \left\|p-p'\right\|_{\infty}\left\|\frac{1}{p}\right\|_{\infty}
\end{equation*}
\item  \label{comp_trans_action}
Let $(\omega,q,y) \in \Omega \times \Delta(L) \times Y^\eta$. Then
\begin{equation*}
\left|g^\eta(p',q,\omega,x',y)-g^\eta(p,q,\omega,x,y)\right| \leq \left\|g\right\|_{\infty} \left\|p'-p\right\|_1
\end{equation*}
\end{enumerate}
\end{proposition}
\begin{proof}
\begin{enumerate}[(i)]
\item
For all $i \in I$, we have
\begin{eqnarray*}
\overline{x'}^{p'}(i)&=&\sum_{k \in K} x'(i|k)p'(k)
\\
&=& \sum_{k \in K} \frac{\bar{x}^p(i)}{p'(k)} \left[{p}'(k)+p^x(k | i)-p(k)\right] p'(k)
\\
&=& \bar{x}^p(i). 
\end{eqnarray*}
\item
Let $i \in I$ and $k \in K$. We have
\begin{eqnarray*}
{p'}^{x'}(k|i)&=& \frac{\frac{\bar{x}^p(i)}{p'(k)} \left[{p}'(k)+p^x(k | i)-p(k)\right] p'(k)}{\bar{x}^p(i)}
\\
&=& {p}'(k)+p^x(k | i)-p(k),
\end{eqnarray*}
and the result follows.
\item
When $(x,p,p') \notin X_0$, the result is obvious. Assume $(x,p,p') \in X_0$. Let $(k,i) \in K \times I$. If $\bar{x}^p(i)=0$, then $x'(i|k)=x(i|k)=0$. Assume $\bar{x}^p(i) \neq 0$. We have
 \begin{eqnarray*}
 |x'(i|k) - x(i|k)|&=& \bar{x}^p(i) \left|\frac{{p'}^{x'}(k|i)}{p'(k)}-\frac{p^{x}(k|i)}{p(k)}\right|
 \\
 &=& \bar{x}^p(i) \left|\frac{{p'}^{x'}(k|i)-p^{x}(k|i)}{p(k)}
+ {p'}^{x'}(k|i)\left(\frac{1}{p'(k)}-\frac{1}{p(k)}\right)\right| 
\\
&=&
\left|\bar{x}^p(i) \frac{p'(k)-p(k)}{p(k)}
+ x'(i|k)\frac{p(k)-p'(k)}{p(k)}\right| 
\\
&\leq&  \left\|p-p'\right\|_{\infty} \left\|\frac{1}{p}\right\|_{\infty}.
 \end{eqnarray*}
 \item
 When $(x,p,p') \notin X_0$, the result is straightforward. Assume $(x,p,p') \in X_0$.
 \begin{eqnarray*}
\left|g^\eta(p',q,\omega,x',y)-g^\eta(p,q,\omega,x,y)\right| &=& 
\sum_{k,\ell,i,j} y(j|\ell)q(\ell)|x'(i|k)p'(k)-x(i|k)p(k)|g(k,\ell,\omega,i,j)
\\
&\leq& \left\|g\right\|_{\infty} \sum_{(k,i) \in K \times I} \left|x'(i|k)p'(k)-x(i|k)p(k) \right|
\\
&=& \left\|g\right\|_{\infty} \sum_{(k,i) \in K \times I} \bar{x}^p(i) \left|p'(k)-p(k)\right|
\\
&=& \left\|g\right\|_{\infty}\left\|p'-p\right\|_1. 
\end{eqnarray*}
\end{enumerate}
\end{proof}
The following figure illustrates Properties \ref{action_conservation} and \ref{error_conservation}, in the case of a type set $K:=\left\{k_1,k_2,k_3\right\}$ and an action set $I:=\left\{A,B,C\right\}$. The figure centered in $p'$ is the translation of the figure centered in $p$, hence the name of the mapping $T$. 
\begin{center}
\includegraphics[scale=0.99]{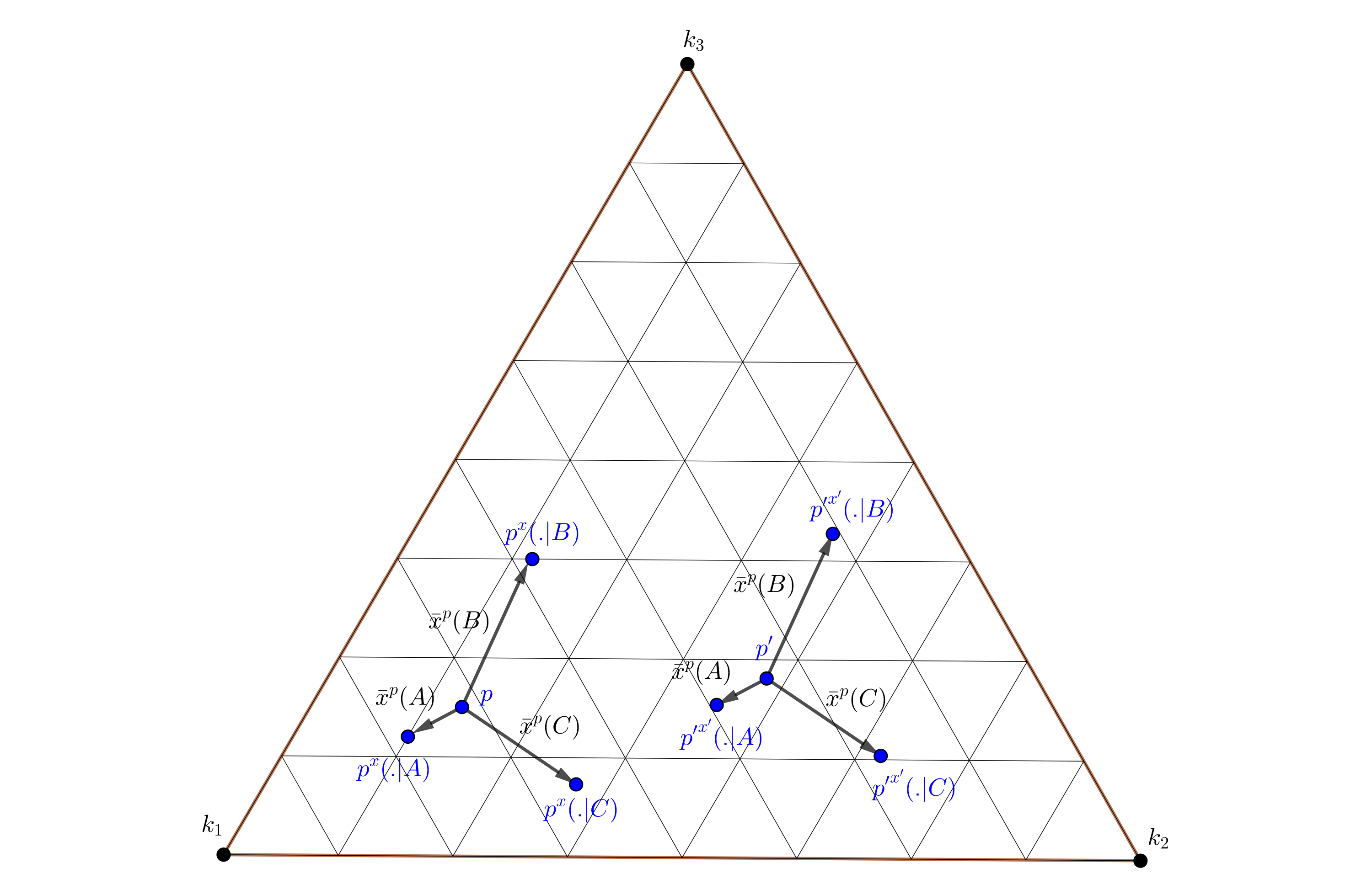}
\end{center}
\subsection{Proof of Proposition \ref{comparison_e_f_w}}
We are now ready to complete the proof of Theorem \ref{main_thm}, by showing Proposition \ref{comparison_e_f_w}. Fix $\Gamma=(K,L,\Omega,I,J,\rho,g)$ an absorbing game with incomplete information on both sides satisfying Assumption \ref{asslip}. 
We proceed by induction on $N= \supp(p) \geq 1$. Consider the case $N=1$, that is, $p=\delta_k$, for some $k \in K$. In this case, the result holds for $\varepsilon=0$ and any triangulation. Indeed, let $(P,Q)$ be a triangulation, $\lambda \in (0,1]$ and $(p,q,\omega) \in \Omega^f$. Let $\sigma$ be an optimal stationary strategy for Player 1 in $\Gamma^{\eta}_{\lambda}$, and $\tau'$ be an optimal strategy of Player 2 in $\Gamma_\lambda^{f}$. The strategy $\sigma$ can be seen as a strategy of $\Gamma^{f}$. Moreover, define a stationary strategy $\tau \in T^{\eta}$ by for all $(p,q,\omega) \in \Omega^f$, $\tau(p,q,\omega):=(\tau'(p,q,\omega),F)$, and $\tau(p,q,\omega)$ is arbitrary otherwise. 
We have $v^{\eta}_\lambda(p,q,\omega) \leq \gamma^{\eta}_{\lambda}(p,q,\omega,\sigma,\tau)=\gamma^{f}_{\lambda}(p,q,\omega,\sigma,\tau') \leq v_\lambda^f(p,q,\omega)$, hence the result holds for $N=1$. 
\\

Assume $N \geq 2$, and that the result holds for $N' \leq N-1$.   Let $\varepsilon \in (0,1/4]$ and $C>0$. Note that an $(\alpha,C)$-triangulation of $\Delta(K)$ induces an $(\alpha,C)$-triangulation on any facet of $\Delta(K)$, and similarly for $\Delta(L)$. Consequently, by induction assumption, there exists $\alpha'>0$ such that for 
any $(\alpha',C)$-triangulation $(P, Q)$ of $\Delta(K) \times \Delta(L)$, for all $p \in F_0=\left\{p \in \Delta(K) \ | \exists k \in K, \  p(k)=0 \right\}$, $q \in \Delta(L)$ and $\omega \in \Omega$, $v^f_\lambda(p,q,\omega) \geq v^{\eta}_\lambda(p,q,\omega)-\varepsilon$. Define
\begin{equation} \label{choice_alpha}
\alpha=\min\left\{\alpha',\frac{\varepsilon^{11}}{12}\left(C+1\right)^{-1} |K|^{-2} \right\}.
\end{equation}

Consider now an $(\alpha,C)$-triangulation $(P,Q)$ of $\Delta(K) \times \Delta(L)$, and $(p,q,\omega) \in \Omega^{f}$. Let us prove that $v^f_\lambda(p,q,\omega) \geq v^{\eta}_\lambda(p,q,\omega)-\Theta(\varepsilon)$, where $\Theta(\varepsilon)$ is a quantity that vanishes as $\varepsilon$ tends to 0. The proof proceeds in three main steps. First, a coupling between histories of $\Gamma^{\eta}$ and $\Gamma^f$ is defined. Second, it is proven that under this coupling, belief dynamics in $\Gamma^{\eta}$ and $\Gamma^f$ remain close until belief on Player 1's type gets very near the frontier of $\Delta(K)$. Third, a dynamic programming argument and the induction hypothesis are used to conclude.   \\
\textbf{Step 1}
Let $\sigma^{\eta} \in \Sigma^\eta$ be given by Proposition \ref{prop:bounded_var_eta}, and $\tau^f$ be a (behavior) optimal strategy of Player 2 in $\Gamma_\lambda^f$. 
We build recursively a process 
\\
$E_m=(X_{m-1},X'_{m-1},Y_{m-1},Y'_{m-1},I_{m-1},J_{m-1},P_{m},P'_{m},Q_{m},\Omega_{m})_{m \geq 1}$ on $X^2 \times Y^{\eta} \times Y \times I \times J \times \Delta(K) \times P \times Q \times \Omega$ 
with law $\m{P}$, in the following way:
\begin{itemize}
\item
For $m=1$, $X_0$, $X'_0$, $Y_0$, $Y'_0$, $I_0$, $J_0$ are arbitrary, $P_1=P'_1=p$, $Q_1=q$, $\Omega_1=\omega$. 
\item
For $ m \geq 2$: 
\begin{itemize} 
\item
$X_m \in X$ is the mixed action prescribed by strategy $\sigma^{\eta}$ in $\Gamma^{\eta}$, given history $H^{\eta}_m:=(\Omega_1,P_1,Q_1,X_1,Y_1,\dots,\Omega_{m-1},P_{m-1},Q_{m-1},X_{m-1},Y_{m-1},P_m,Q_m,\Omega_m)$:
$X_m:=\sigma^{\eta}_m(H^{\eta}_m)$.
\item
The random variable $Y'_m$ is the (realized) mixed action generated by strategy $\sigma^{f}$ in $\Gamma^{f}$, given history $H^f_m:=(\Omega_1,P'_1,Q_1,X'_1,Y'_1,\dots,\Omega_{m-1},P'_{m-1},Q_{m-1},X'_{m-1},Y'_{m-1},\Omega_m,P'_m,Q_m)$: the law of $Y'_m$ conditional to $E_1,\dots,E_m$ is $\tau^f_m(H^f_m)$.
\item
$I_m$ and $J_m$ are random variables representing Players' realized actions:
\begin{equation*}
\forall i \in I \quad \m{P}(I_m=i,J_m=j|E_1,\dots,E_m):=\overline{X_m}^{P_m}(i)\overline{Y_m}^{Q_m}(j).
\end{equation*}
\item
$\Omega_{m+1}$ is the state at stage $m+1$, and is drawn from $\rho(\Omega_m,I_m,J_m)$:
\begin{equation*}
\forall \omega \in \Omega \quad \m{P}(\Omega_{m+1}=\omega|E_1,\dots,E_m,I_m,J_m):=\rho(\omega|\Omega_m,I_m,J_m).
\end{equation*} 
\item
$P_{m+1}$ is the belief over Player 1's type, given that Player 1 played $X_m$ and the realized action is $I_m$:
\begin{equation*}
P_{m+1}:=P_m^{X_m}(.| I_m). 
\end{equation*}
$X'_m \in X$ is the translation of $X_m$, relative to $P_m$ and $P'_m$:
\begin{equation*}
X'_{m}:=T(X_m,P_m,P'_m). 
\end{equation*}
\item
$P'_{m+1}$ (resp., $Q_{m+1}$) is drawn from the splitting of the posterior belief ${P'_m}^{X'_m}(.|I_m)$ (resp., $Q_m^{Y'_m}(.|J_m)$): 
\begin{equation*}
\m{P}(P'_{m+1}=p,Q_{m+1}=q|E_1,\dots,E_m,I_m,J_m)=S[p|{P'_m}^{X'_m}(.|I_m)]S[q|Q_m^{Y'_m}(.|J_m)].
\end{equation*}       
\item
$Y_m:=(Y'_m,F)$. 
\end{itemize}
\end{itemize}

This process defines implicitly a strategy of Player 1 in $\Gamma^f$ and a strategy of Player 2 in $\Gamma^{\eta}$. Indeed, define 
$\sigma^f(H^f_m)$ as being the law of $X'_m$ conditional to $H^f_m$, and $\tau^{\eta}(H^{\eta}_m)$ by the law of $Y_m$ conditional to $H^{\eta}_m$. This defines $\sigma^f$ (resp., $\tau^{\eta}$) on all finite histories that are reached with positive probability by the 
process $(H^{f}_m)$ (resp., $(H^\eta_m)$), and $\sigma^f$ (resp., $\tau^{\eta}$) is defined arbitrarily otherwise. 
By definition, the law of $(H^{\eta}_m)_{m \geq 1}$ is $\m{P}^{\eta}_{p,q,\omega,\sigma^\eta,\tau^\eta}$, and the law of $(H^f_m)_{m \geq 1}$ is $\m{P}^{f}_{p,q,\omega,\sigma^f,\tau^f}$.

The figure below illustrates the coupling dynamics in a game with three types $k_1$, $k_2$ and $k_3$, and three actions, at some stage $m \geq 1$. 
Given stage beliefs $P_m$ and $P'_m$, and a realized action $I_m$, posterior beliefs $P_{m+1}$ and $P'_{m+1}$ are represented. 
\begin{center}
\includegraphics[scale=1.1]{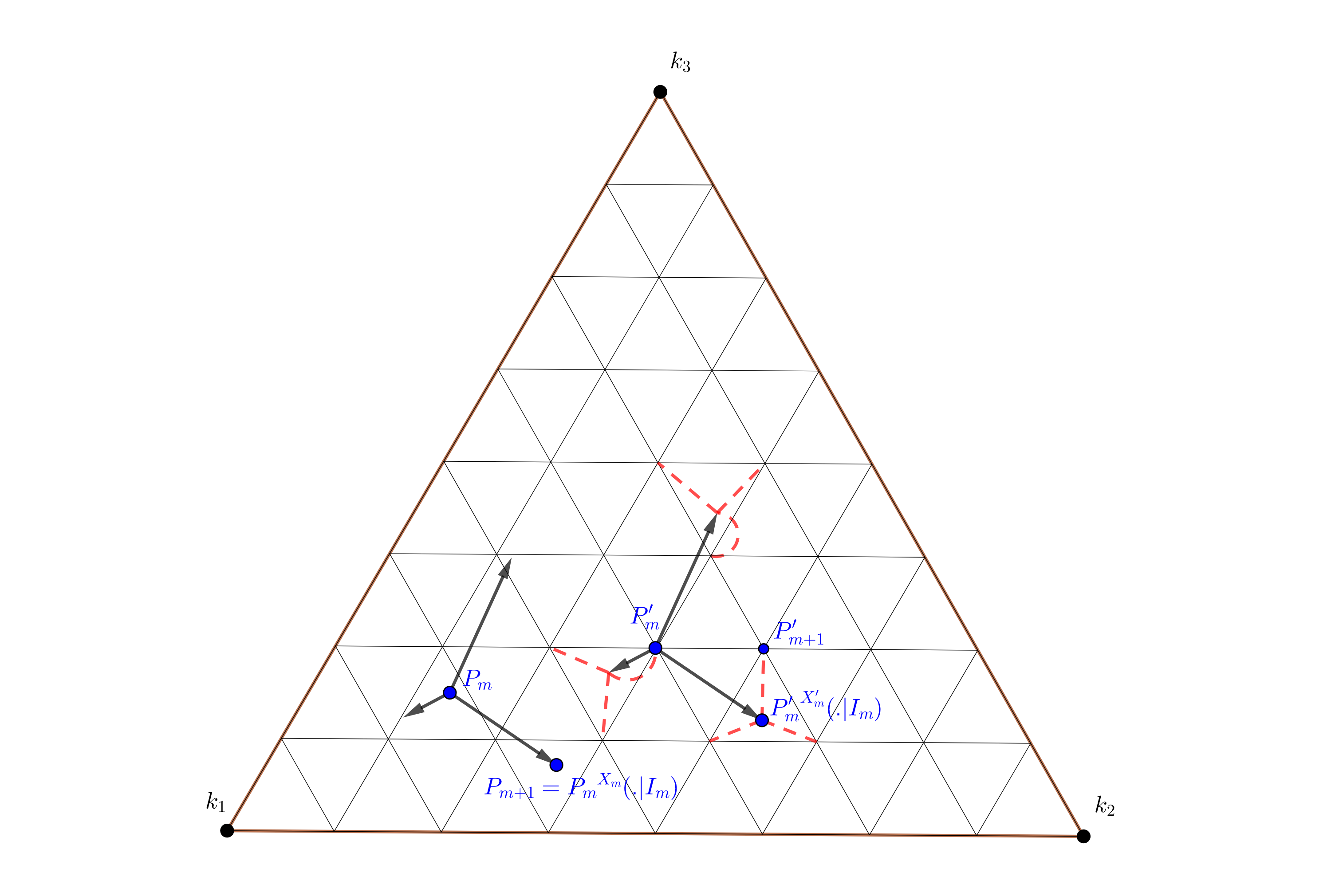}
\end{center}
 \textbf{Step 2: Properties of belief dynamics $(P_m)$ and $(P'_m)$}

Denote by $T_0$ the stopping time 
\\
$$T_0:=\max \left\{m \geq 1 | \forall k \in K, \ \forall i \in I,  \ P'_m(k)+P_m^{X_m}(k|i)+P_m(k) \geq 0, \ P_m(k) \geq \varepsilon \ \text{and}  \ P'_m(k)>0 \right\}+1.$$
We will prove the three following properties:
\begin{enumerate}
\item \label{diff_traj_1}
For all $m \geq 1$, 
\begin{equation*}
\m{E}(\left\|P'_{m \wedge T_0}-P_{m \wedge T_0}\right\|_1) \leq \varepsilon
\end{equation*}
\item
\label{diff_traj_2}
\begin{equation*}
\m{E}(1_{T_0<+\infty}\left\|P'_{T_0}-P_{T_0}\right\|_1) \leq \varepsilon
\end{equation*}
\item
\label{traj_border}
\begin{equation*}
\m{P}(T_0<+\infty, P'_{T_0} \notin F_{2\varepsilon}) \leq \varepsilon
\end{equation*}
\end{enumerate}
Let us first prove Property \ref{diff_traj_1}. 
Notice that on $\left\{m < T_0 \right\}$, by Proposition \ref{prop_translation}, we have
\begin{equation*}
{P'_m}^{X'_m}(.|I_m)-P_{m+1}=P'_m-P_m.
\end{equation*}
It follows that for any $t \geq 1$,
\begin{eqnarray*}
P'_{t \wedge T_0}-P_{t \wedge T_0} &=& \sum_{m=1}^{t \wedge T_0-1} [(P'_{m+1}-P_{m+1})-(P'_{m}-P_{m})]
\\
&=& \sum_{m=1}^{t \wedge T_0-1} [P'_{m+1}-{P'_{m}}^{X'_{m}}(.|I_{m})].
\end{eqnarray*}
Set $Z_m:=1_{m < T_0} \left[ P'_{m+1}-{P'_{m}}^{X'_{m}}(.|I_{m})\right]$, so that $P'_{t \wedge T_0}-P_{t \wedge T_0}=\sum_{m=1}^{t-1} Z_{m}$. 
Denote by $\mathcal{F}_m$ the $\sigma$-field generated by $(H^{\eta}_m,H^f_m,(I_{m'})_{m' \leq m})$. Note that $T_0$ is a stopping time with respect to $(\mathcal{F}_m)$. Moreover, 
because the law of $P'_{m+1}$ conditional to $\mathcal{F}_m$ is $S[{P'_m}^{X'_m}(.|I_m)]$, we have
\begin{equation*}
\m{E}(Z_{m}| \mathcal{F}_m)=1_{m<T_0}  \m{E} \left( P'_{m+1}-{P'_{m}}^{X'_{m}}(.|I_{m})|  \mathcal{F}_m \right)=0.
\end{equation*}
Moreover, for all $m<m'$, 
\begin{eqnarray*}
\m{E}(Z_m \cdot Z_{m'})&=& \m{E}(Z_{m} \cdot \m{E}( Z_{m'}|\mathcal{F}_{m'}))=0,
\end{eqnarray*}
 hence 
\begin{eqnarray}
\m{E}\left(\left\|P'_{t \wedge T_0}-P_{t \wedge T_0}\right\|_2^2 \right)=\m{E}\left(\sum_{m=1}^{t-1} \left\|Z_{m} \right\|_2^2 \right). \label{eq_pyth}
\end{eqnarray}
Let $m \in \left\{1,\dots,t-1\right\}$, and $s$ be the stepsize of the $(\alpha,C)$-triangulation $(P,Q)$. We have almost surely
\begin{eqnarray*}
\m{E}(\left\|Z_m\right\|_2^2| \mathcal{F}_m)&=& 1_{m < T_0} \m{E}\left(\left\|P'_{m+1}-{P'_m}^{X'_m}(.|I_m)\right\|_2^2 | \mathcal{F}_m\right)
\\
&=& 1_{m < T_0} \m{E}\left(1_{P'_{m+1}=P'_m} \left\|P'_{m}-{P'_m}^{X'_m}(.|I_m) \right\|_2^2| \mathcal{F}_m\right)
\\
&+&1_{m < T_0} \m{E}\left(1_{P'_{m+1} \neq P'_m}\left\|P'_{m+1}-{P'_m}^{X'_m}(.|I_m) \right\|_2^2 | \mathcal{F}_m \right) 
\\
&\leq& 1_{m < T_0} \left[s \m{E}\left( 1_{P'_{m+1}=P'_{m}} \left\|P'_{m}-{P'_m}^{X'_m}(.|I_m)\right\|_2 | \mathcal{F}_m\right)+s^2 \m{P}(P'_{m+1} \neq P'_m | \mathcal{F}_m)\right]
\\
&\leq& 1_{m < T_0} \left[s \m{E}\left( \left\|P'_{m}-{P'_m}^{X'_m}(.|I_m)\right\|_2 | \mathcal{F}_m \right)+s^2 \m{P}(P'_{m+1} \neq P'_m | \mathcal{F}_m)\right].
\end{eqnarray*} 
Let $A:=S[P'_m|{P'_m}^{X'_m}(.|I_m)]$. Because $P$ is an $(\alpha,C)$-triangulation, we have
\begin{equation*}
1-A \leq C s^{-1} \cdot \left\|P'_{m}-{P'_m}^{X'_m}(.|I_m) \right\|_2
\end{equation*}
We deduce that
\begin{eqnarray*}
\m{P}(P'_{m+1} \neq P'_m | \mathcal{F}_m)=1-A \leq C s^{-1} \cdot \m{E}\left(\left\|P'_{m}-{P'_m}^{X'_m}(.|I_m) \right\|_2 | \mathcal{F}_m\right) 
\end{eqnarray*}
and thus
\begin{eqnarray*}
\m{E}(\left\|Z_m\right\|_2^2 | \mathcal{F}_m) &\leq& 1_{m < T_0}\left(C+1\right) s \cdot \m{E}\left(\left\|P'_{m}-{P'_m}^{X'_m}(.|I_m) \right\|_2| \mathcal{F}_m\right),
\\
&=&1_{m < T_0}\left(C+1\right)s \cdot \m{E}(\left\|P_{m+1}-P_m\right\|_2 | \mathcal{F}_m),
\end{eqnarray*}
where we used Proposition \ref{prop_translation} in the above equality. We deduce that
\begin{eqnarray} \label{ineq_Z}
\m{E}(\left\|Z_m\right\|_2^2) \leq \left(C+1 \right) s \cdot \m{E}(1_{m < T_0} \left\|P_{m+1}-P_m\right\|_2).
\end{eqnarray}
Define $T:=\max\left\{m \geq 1, P_m \in \Delta(K) \setminus F_{\varepsilon} \right\}$. We have
\begin{align*} 
\m{E}(\left\|P'_{t \wedge T_0}-P_{t \wedge T_0}\right\|_2^2) &\leq \left(C+1\right)\alpha \m{E}\left(\sum_{m=1}^{t-1} 1_{m < T_0} \left\|P_{m+1}-P_m\right\|_2 \right) & ;
\text{Eq.} \ (\ref{eq_pyth}) \ \text{and} \ (\ref{ineq_Z}) ; \ s \leq \alpha
\\
&\leq  \left(C+1\right)\alpha \m{E}\left(\sum_{m=1}^{T}\left\|P_{m+1}-P_m\right\|_2 \right) & ; T_0 \leq T+1
\\
&\leq  \left(C+1\right)\alpha 3 |K| \varepsilon^{-5} & ; \text{Prop. \ref{prop:bounded_var_eta}}.
\\
&\leq \frac{\varepsilon^6}{4|K|} & ; \text{Eq.} \ (\ref{choice_alpha}) \numberthis \label{bound_variance}
\end{align*}
Hence, 
\begin{align*}
\m{E}(\left\|P'_{t \wedge T_0}-P_{t \wedge T_0}\right\|_1) &\leq \sqrt{|K|} \m{E}(\left\|P'_{t \wedge T_0}-P_{t \wedge T_0}\right\|_2) 
\\
&\leq \sqrt{|K|} \left[\m{E}(\left\|P'_{t \wedge T_0}-P_{t \wedge T_0}\right\|^2_2)\right]^{\frac{1}{2}}
\\
&\leq \varepsilon^3/2 \numberthis \label{error_1},
\end{align*}
and since $\varepsilon^3/2 \leq \varepsilon$, Property \ref{diff_traj_1} is proved. 
It follows by the dominated convergence theorem that 
\begin{eqnarray*}
\m{E}(1_{T_0<+\infty}\left\|P'_{T_0}-P_{T_0}\right\|_1) \leq \varepsilon^3/2,
\end{eqnarray*}
which proves Property \ref{diff_traj_2}. In addition,
using Markov inequality, we deduce that
\begin{align*}
\m{P}(\left\{T_0<+\infty\right\}, \left\|P'_{T_0}-P_{T_0}\right\|_1 \geq \varepsilon^2) 
&\leq \varepsilon^{-2} \m{E}(1_{T_0<+\infty}\left\|P'_{T_0}-P_{T_0}\right\|_1)
\\
&\leq \varepsilon/2 \numberthis \label{ineq:frontier}
\end{align*}
Moreover, by definition of $T_0$, on the event $\left\{T_0<+\infty \right\}$, there exists random elements $i \in I$ and $k \in  K$ such that either
\\
$P'_{T_0}(k)+P_{T_0}^{X_{T_0}}(k|i)-P_{T_0}(k) <0$, or $P_{T_0}(k)<\varepsilon$, or $P'_{T_0}(k)=0$. 
Because $\sigma^{\eta}$ is $\varepsilon$-ambiguous, we have on the event $\left\{T_0<+\infty \right\}$
\begin{equation*}
P_{T_0}^{X_{T_0}}(k|i) \geq \varepsilon P_{T_0}(k).
\end{equation*}
Thus, on the event $\left\{T_0<+\infty \right\}$, when $P_{T_0}(k) \geq \varepsilon$, we have either
 $P'_{T_0}(k)-P_{T_0}(k)<-{P_{T_0}}^{X_{T_0}}(k|i) \leq  -\varepsilon^2$, or $P'_{T_0}(k)=0 \leq P_{T_0}(k)-\varepsilon$. 
 Consequently, on $\left\{T_0<+\infty \right\}$, either $\left\|P'_{T_0}-P_{T_0}\right\|_1>\varepsilon^2$ or $P_{T_0} \in F_{\varepsilon}$. Hence $\left\{T_0<+\infty \right\} \subset \left\{\left\|P'_{T_0}-P_{T_0}\right\|_1>\varepsilon^2\right\} \cup \left\{P_{T_0} \in F_{\varepsilon}\right\}$. We deduce that
\begin{align*}
\m{P}(T_0<+\infty, P_{T_0} \notin F_{\varepsilon}) &\leq \m{P}(T_0<+\infty, \left\|P'_{T_0}-P_{T_0}\right\|_1> \varepsilon^2) 
\\
&\leq \varepsilon/2 & ; \text{Eq. \eqref{ineq:frontier}} 
\end{align*}
It follows that
\begin{equation*}
\m{P}(T_0<+\infty, P'_{T_0} \notin F_{2\varepsilon}) \leq \m{P}(T_0<+\infty, P_{T_0} \notin F_{\varepsilon} \ \text{or} \left\|P'_{T_0}-P_{T_0}\right\|_1> \varepsilon^2) \leq \varepsilon/2+\varepsilon/2=\varepsilon,
\end{equation*}
and Property \ref{traj_border} is proved. 

\textbf{Step 3: Conclusion}
\\
By optimality of $\tau^f$ in $\Gamma_\lambda^f$, we have
\begin{eqnarray*}
v^f_{\lambda}(p,q,\omega) &\geq& 
\m{E}\left(\sum_{m=1}^{T_0-1} \lambda(1-\lambda)^{m-1} g^f(P'_m,Q_m,\Omega_m,X'_m,Y'_m)\right)+\m{E}(1_{T_0<+\infty}(1-\lambda)^{T_0-1} v^f_{\lambda}(P'_{T_0},Q_{T_0},\Omega_{T_0}))
\end{eqnarray*}
Let us bound from below the left-hand side term. We have
\begin{align*}
&\m{E}(1_{m <T_0}|g^f(P'_m,Q_m,\Omega_m,X'_m,Y'_m)-g^{\eta}(P_m,Q_m,\Omega_m,X_m,Y_m)|) 
\\
&\leq
\left\|g\right\|_{\infty} \m{E}(1_{m <T_0}\left\|P'_m-P_m\right\|_1) & ; \text{Proposition \ref{prop_translation} \ref{comp_trans_action}}
\\
&\leq  \varepsilon \left\|g\right\|_{\infty} & ; \text{Property \ref{diff_traj_1}}
\end{align*}
We deduce that
\begin{eqnarray}
&&\m{E}\left(\sum_{m=1}^{T_0-1} \lambda(1-\lambda)^{m-1}g^f(P'_m,Q_m,\Omega_m,X'_m,Y'_m)\right) 
\nonumber \\
&\geq& \m{E}\left(\sum_{m=1}^{T_0-1} \lambda(1-\lambda)^{m-1}g^{\eta}(P_m,Q_m,\Omega_m,X_m,Y_m)\right)
-  \varepsilon\left\|g\right\|_{\infty}. \label{eq_diff_payoff}
\end{eqnarray}
As for the right-hand side term, we have  
\begin{align*}
&\m{E}\left(1_{T_0<+\infty}(1-\lambda)^{T_0-1}1_{P'_{T_0} \in F_{2\varepsilon}} \left[v^f_{\lambda}(P'_{T_0},Q_{T_0},\Omega_{T_0})-v^{\eta}_{\lambda}(P'_{T_0},Q_{T_0},\Omega_{T_0})\right]\right)
\\
&\geq -\left\|v^{\eta}_\lambda-v^f_{\lambda}\right\|_{F_0} -9 \varepsilon \left\|g\right\|_{\infty} & ; \text{Proposition \ref{lipf}}, \ \alpha \leq \varepsilon
\\
&\geq -\varepsilon -9 \varepsilon \left\|g\right\|_{\infty} & ; \text{Induction hypothesis}
\end{align*}
Moreover, we have 
\begin{align*}
&\m{E}\left(1_{T_0<+\infty}(1-\lambda)^{T_0-1} 1_{P'_{T_0} \notin F_{2\varepsilon}} \left[v^f_{\lambda}(P'_{T_0},Q_{T_0},\Omega_{T_0})-v^{\eta}_{\lambda}(P'_{T_0},Q_{T_0},\Omega_{T_0})\right]\right)
\\
&\geq
-\m{P}(T_0<+\infty, P'_{T_0} \notin F_{2\varepsilon}) 2 \left\|g\right\|_{\infty} & ; \left\|v^f_\lambda\right\|_{\infty}, \ \left\|v^{\eta}_\lambda\right\|_{\infty} \leq \left\|g\right\|_{\infty} 
\\
&\geq -2\varepsilon \left\|g\right\|_{\infty} & ; \text{Property } \ref{traj_border}
\end{align*}
We deduce that 
\begin{eqnarray}
\m{E}\left(1_{T_0<+\infty}(1-\lambda)^{T_0-1}\left[v^f_{\lambda}(P'_{T_0},Q_{T_0},\Omega_{T_0})-v^{\eta}_{\lambda}(P'_{T_0},Q_{T_0},\Omega_{T_0})\right]\right)
\geq -\varepsilon -11 \varepsilon \left\|g\right\|_{\infty}, \label{eq_diff_value_1}
\end{eqnarray}
Moreover, applying the Lipschitz property of $v^{\eta}_\lambda$ and then Property \ref{diff_traj_2} yields
\begin{align}
\m{E}\left(1_{T_0<+\infty}(1-\lambda)^{T_0-1}\left[v^\eta_{\lambda}(P'_{T_0},Q_{T_0},\Omega_{T_0})-v^{\eta}_{\lambda}(P_{T_0},Q_{T_0},\Omega_{T_0})\right]\right) &\geq -\m{E}(1_{T_0<+\infty} \left\|P'_{T_0}-P_{T_0}\right\|_1) \left\|g\right\|_{\infty} \nonumber
\\
&\geq -\varepsilon \left\|g\right\|_{\infty}. \label{eq_diff_value_2}
\end{align}

We deduce that
\begin{align*}
v^f_{\lambda}(p,q,\omega) &\geq
\m{E}\left(\sum_{m=1}^{T_0-1} \lambda(1-\lambda)^{m-1} g^\eta (P_m,Q_m,\Omega_m,X_m,Y_m)\right)&+
\m{E}(1_{T_0<+\infty}(1-\lambda)^{T_0-1} v^{\eta}_{\lambda}(P_{T_0},Q_{T_0},\Omega_{T_0})) 
\\
&
- \varepsilon-13\varepsilon \left\|g\right\|_{\infty} & ; Eq. \ \eqref{eq_diff_payoff},  \eqref{eq_diff_value_1}, \eqref{eq_diff_value_2}
\\
&\geq
v^{\eta}_{\lambda}(p,q,\omega)-\varepsilon -25 \varepsilon \left\|g\right\|_{\infty} & ; \text{Shapley eq. for $v_\lambda^{\eta}$ and $12\left\|g\right\|_\infty$-optimality of $\sigma^\eta$}
\end{align*}
which concludes the proof. 
\subsection{Missing proofs} \label{subsec:missing}
\subsubsection{Proof of Proposition \ref{equal_e_eta}}
\begin{proof}
We are going to prove that $v^e_\lambda$ satisfies the same Shapley equation as $v^\eta_\lambda$. 
Recall that the function $v^e_\lambda:\Omega^e \rightarrow \m{R}$ is the unique solution of the Shapley equation:
\begin{equation} \label{shapley_e}
\forall \, \omega^e \in \Omega^e, \ v^e_\lambda(\omega^e)=\displaystyle \val_{(x,y) \in X \times Y} \left\{\lambda g^e(\omega^e,x,y)+(1-\lambda) \m{E}^e_{\omega^e,x,y}(v^e_\lambda) \right\}.
\end{equation}
 Similarly, $v^{\eta}_\lambda$ is the unique solution of the Shapley equation
 \begin{equation} \label{shapley_eta}
\forall \, \omega^{e} \in \Omega^{e}, \quad v^{\eta}_\lambda(\omega^e)=\val_{(\mu,\nu) \in \Delta(X) \times \Delta(Y^{\eta})} \left\{\lambda g^\eta(\omega^{e},\mu,\nu)+(1-\lambda) \m{E}^{\eta}_{\omega^e,\mu,\nu}(v^{\eta}_\lambda) \right\},
\end{equation}
where for $f : \Omega^e \rightarrow \m{R}$ and $(\omega^{e},x,y) \in \Omega^e \times X \times Y^{\eta}$,
\begin{eqnarray*}
\m{E}^{\eta}_{\omega^{e},x,y}(f)&:=&\sum_{{\omega^e}' \in \Omega^{e}} \rho^{\eta}({\omega^{e}}',x,y|\omega^{e},x,y) f({\omega^e}'),
\end{eqnarray*}
and $(x,y) \rightarrow \m{E}^{\eta}_{\omega^e,x,y}(f)$ and $(x,y) \rightarrow g^\eta(\omega^{e},x,y)$ are extended linearly to $\Delta(X) \times \Delta(Y^{\eta})$. Considering distributions over $X$ and $Y^{\eta}$ is necessary because we do not know \textit{a priori} whether $\Gamma_\lambda^{\eta}$ has a value in pure strategies or not. 

Consider $(\omega^{e},x,y) \in \Omega^e \times X \times Y^{\eta}$. We have
\begin{equation*}
\m{E}^{\eta}_{\omega^{e},x,y}(v^{e}_\lambda)=1_{y_2=E}\m{E}^{e}_{\omega^{e},x,y_1}(v^e_\lambda) +1_{y_2=F} \sum_{(p',q',q'',\omega') \in \Delta(K) \times Q \times \Delta(L) \times \Omega} S[q'|q''] \rho^e(p',q'',\omega'|\omega^e,x,y) v^e_{\lambda}(p',q',\omega')
\end{equation*}
Let $(p',q',\omega') \in \Omega^e$. 
Because $v^e_\lambda$ is convex in $q$, we have 
\begin{equation*}
v^e_{\lambda}(p',q'',\omega') \leq \sum_{q' \in Q} S[q'|q''] v^e_\lambda(p',q',\omega'),
\end{equation*}
which implies that 
\begin{equation*}
\m{E}^e_{\omega^e,x,y_1}(v^e_\lambda) \leq \m{E}^{\eta}_{\omega^e,x,y}(v^e_\lambda).
\end{equation*}
We deduce that
\begin{eqnarray*}
\val_{(x,y) \in X \times Y} \left\{\lambda g^e(\omega^e,x,y)+(1-\lambda) \m{E}^{e}_{\omega^e,x,y}(v^e_\lambda) \right\}&=&\val_{(x,y) \in X \times Y^{\eta}} \left\{\lambda g^\eta(\omega^e,x,y)+(1-\lambda) \m{E}^{\eta}_{\omega^e,x,y}(v^e_\lambda) \right\}
\\
&=&\val_{\Delta(X) \times \Delta(Y^{\eta})} \left\{\lambda g^\eta(\omega^e,\mu,\nu)+(1-\lambda) \m{E}^\eta_{\omega^e,\mu,\nu}(v^e_\lambda) \right\},
\end{eqnarray*}
where the last line stems from the fact that when a game has a value in pure strategies, the value in mixed strategies is identical. Combining with (\ref{shapley_e}), we get that $v^e_\lambda$ satisfies the functional equation $(\ref{shapley_eta})$, hence by uniqueness $v_\lambda^e=v_\lambda^{\eta}$. Moreover, the Shapley equation in $\Gamma_\lambda^{\eta}$ can be written in pure strategies, hence $\Gamma_\lambda^{\eta}$ admits pure optimal stationary strategies. 
\end{proof}
\subsubsection{Adapting Proposition \ref{opt_concise} to $\Gamma^\eta$}

As explained in Subsection \ref{subsec:modif}, the only ingredient missing for the proof of Proposition \ref{prop:bounded_var_eta} is the validity of Proposition \ref{opt_concise} in the game $\Gamma^{\eta}$, that is: 
\begin{proposition} 
Assume that $\Gamma$ is an absorbing game with incomplete information on both sides. Let $\varepsilon>0$ and $\lambda \in (0,1]$. Let $\sigma$ be a (pure) optimal stationary strategy in $\Gamma^\eta_\lambda$, and $\sigma'$ a (pure) stationary strategy such that for all $\omega^e=(p,q,\omega) \in \Omega^e$, $\sigma'(\omega^e)$ is an $\varepsilon$-convexification of $\sigma(\omega^e)$ at $p$. Then $\sigma'$ is $2\varepsilon \left\|g\right\|_{\infty}$-optimal in $\Gamma^\eta_\lambda$. 
\end{proposition}
\begin{proof}
This is essentially the same proof than in Proposition \ref{opt_concise}. Let $\tau: \Omega^e \rightarrow Y^{\eta}$ be a (pure) stationary strategy for Player 2. 
Let $p,q$ that minimize $(p,q) \rightarrow \gamma^{\eta}_\lambda(p,q,\omega^0,\sigma',\tau)-v^\eta_\lambda(p,q,\omega^0)$, and  call this value $D$. In order to prove our result, it is enough to prove that 
$D \geq -2\varepsilon \left\|g\right\|_{\infty}$. Let $(p,q,\omega) \in \Omega^e$, and let $x=\sigma(p,q,\omega)$, $x'=\sigma'(p,q,\omega)$ and $y=\tau(p,q,\omega)$. 
We have 
\begin{equation} \label{dyn_succinct_eta}
\gamma^{\eta}_{\lambda}(p,q,\omega^0,\sigma',\tau)=\lambda g^{\eta}(p,q,\omega^0,x',y)+(1-\lambda) \m{E}^{\eta}_{p,q,\omega^0,x',y}(\gamma^{\eta}_{\lambda}(\, . \, ,\sigma',\tau))
\end{equation}
and 
\begin{equation} \label{dyn_opt_eta}
v^\eta_{\lambda}(p,q,\omega^0) \leq \lambda g^{\eta}(p,q,\omega^0,x,y)+(1-\lambda) \m{E}^{\eta}_{p,q,\omega^0,x,y}(v^\eta_{\lambda})
\end{equation}
Inequality (\ref{payoff_succinct}) extends straigthforwardly to $g^\eta$, hence: 
\begin{equation} \label{payoff_succinct_eta}
g^\eta(p,q,\omega^0,x',y) \geq  g^\eta(p,q,\omega^0,x,y)- \varepsilon \left\|g\right\|_{\infty}. 
\end{equation}
Let us focus on the second term. First, assume that $y_2=F$. For all $(p,q,\omega) \in \Omega^e$, define 
\begin{equation*}
u(p,q,\omega):= \sum_{q' \in Q} S[q'|q]\gamma^{\eta}_\lambda(p,q',\omega,\sigma',\tau),
\end{equation*}
and 
\begin{equation*}
v(p,q,\omega):= \sum_{q' \in Q} S[q'|q] v^{\eta}_\lambda(p,q',\omega). 
\end{equation*}
We have 
\begin{equation*}
\m{E}^{\eta}_{p,q,\omega^0,x',y}(\gamma^\eta_{\lambda}(\, . \, ,\sigma',\tau))=
\m{E}^{e}_{p,q,\omega^0,x',y_1}(u),
\end{equation*}
and 
\begin{equation*}
\m{E}^{\eta}_{p,q,\omega^0,x,y}(v^\eta_{\lambda})=
\m{E}^{e}_{p,q,\omega^0,x,y_1}(v). 
\end{equation*}
 Applying Lemma \ref{lemma:future_payoff} with $u=\gamma^\eta_\lambda(.,\sigma',\tau)$, $v=v^\eta_\lambda$ and $C=\left\|g\right\|_\infty$, and using the same notations for $D$ and $R$, we get
\begin{equation*}
\m{E}^{\eta}_{p,q,\omega^0,x',y}(\gamma^\eta_{\lambda}(\, . \, ,\sigma',\tau))-\m{E}^\eta_{p,q,\omega^0,x,y}(v^{\eta}_{\lambda}) \geq DR -2(1-R) \varepsilon \left\|g\right\|_{\infty}.
\end{equation*}
Combining with (\ref{dyn_succinct_eta}), \eqref{dyn_opt_eta} and (\ref{payoff_succinct_eta}), we get
\begin{eqnarray*}
\gamma^{\eta}_{\lambda}(p,q,\omega^0,\sigma',\tau)-v^{\eta}_\lambda(p,q,\omega^0) &\geq& -\lambda  \varepsilon \left\|g\right\|_\infty+(1-\lambda)(DR -2(1-R) \varepsilon \left\|g\right\|_{\infty})
\\
&=& -\lambda  \varepsilon \left\|g\right\|_\infty+(1-\lambda)(D-(1-R)(D+2\varepsilon \left\|g\right\|_{\infty})),
\end{eqnarray*}
When $\tau_2=E$, this inequality holds too (it corresponds to the proof of inequality (\ref{ineq:xx'}) in Proposition \ref{opt_concise}). 
This implies by definition of $D$
\begin{equation*}
D \geq -\lambda  \varepsilon \left\|g\right\|_\infty+(1-\lambda)(D-(1-R)(D+2\varepsilon \left\|g\right\|_{\infty})),
\end{equation*}
and we conclude as in Proposition \ref{opt_concise} that $D \geq -2\varepsilon \left\|g\right\|_{\infty}$. 
\end{proof}
\subsubsection{Proof of Proposition \ref{lipf}}
\begin{proof}
By assumption, there exists $k \in K$ such that $p(k) \leq \varepsilon$. Hence, the belief $p$ can be decomposed as $p=p(k) \cdot \delta_k+(1-p(k)) \cdot \hat{p}$, with $\hat{p} \in F_0$ and $\left\|p-\hat{p}\right\|_{1}= 2p(k) \leq 2\varepsilon$.
Let $x \in X$ that satisfies
\begin{equation} \label{shapley_prime}
\forall y \in Y^\eta, \quad v_\lambda^{\eta}(\hat{p},q,\omega) \leq \lambda g^\eta(\hat{p},q,\omega,x,y)+(1-\lambda) \m{E}^{\eta}_{\hat{p},q,\omega,x,y} (v^{\eta}_\lambda),
\end{equation}
and for all $k' \in \supp(\hat{p})$, $x(i^*|k')=0$, where $i^*$ is given by Assumption \ref{asslip}. 
Existence of such an $x$ stems from the Shapley equation in $\Gamma^{\eta}_\lambda$. Define $x'$ such that for all $k' \neq k$, $x'(.|k')=x(.|k')$, and 
$x'(i^*|k)=1$. Let $y \in Y$ such that
\begin{equation} \label{ineq_f}
v^f_{\lambda}(p,q,\omega) \geq \lambda g^f(p,q,\omega,x',y)+(1-\lambda)\m{E}^f_{p,q,\omega,x',y}(v^{f}_\lambda) 
\end{equation}
We have
\begin{eqnarray}
g^f(p,q,\omega,x',y)&=&p(k) g^f(k,q,\omega,i^*,y)+(1-p(k)) g^f(\hat{p},q,\omega,x,y) \nonumber
\\
&\geq& -p(k) \left\|g\right\|_\infty+(1-p(k)) g^\eta(\hat{p},q,\omega,x,y). \label{ineq_g}
\end{eqnarray}
Moreover, we have
\begin{eqnarray} 
\m{E}^f_{p,q,\omega,x',y}(v^{f}_\lambda)=
\m{E}^\eta_{p,q,\omega,x',(y,F)}\left(u_\lambda \right),
\end{eqnarray}
where 
\begin{equation*}
\forall (\tilde{p},\tilde{q},\tilde{\omega}) \in \Omega^e, \quad u_\lambda(\tilde{p},\tilde{q},\tilde{\omega}):=\sum_{p' \in P}  S[p'|\tilde{p}] v^{f}_\lambda(p',\tilde{q},\tilde{\omega}).
\end{equation*}
Since
$$
 p^{x'}(.|i) := \left\{
    \begin{array}{ll}
      \hat{p}^x(.|i)  & \mbox{when} \ i \neq i^* \\
      \delta_k  & \mbox{otherwise},
    \end{array}
\right.
$$
we deduce that
\begin{eqnarray} \label{eq:future_payoff_f}
\m{E}^f_{p,q,\omega,x',y}(v^{f}_\lambda) &\geq& -p(k) \left\|g\right\|_\infty +(1-p(k)) 
\m{E}^\eta_{\hat{p},q,\omega,x,(y,F)}\left(u_\lambda \right).
\end{eqnarray}
For all $(\tilde{p},\tilde{q},\tilde{\omega}) \in F_0 \times \Delta(K) \times \Delta(L)$, we have 
\begin{eqnarray*}
u_\lambda(\tilde{p},\tilde{q},\tilde{\omega}) &\geq& \sum_{p' \in P}  S[p'|\tilde{p}] v^{\eta}_\lambda(p',\tilde{q},\tilde{\omega})-\left\|v^f_{\lambda}-v^{\eta}_{\lambda}\right\|_{F_0}
\\
&\geq& v^\eta_\lambda(\tilde{p},\tilde{q},\tilde{\omega})-\alpha \left\|g\right\|_\infty-\left\|v^f_{\lambda}-v^{\eta}_{\lambda}\right\|_{F_0},
\end{eqnarray*}
where we used the fact that $v^\eta_\lambda(.,\tilde{q},\tilde{\omega})$ is $\left\|g\right\|_\infty$-Lipschitz in the second inequality. Combining with \eqref{eq:future_payoff_f}, we get
\begin{equation}
\m{E}^f_{p,q,\omega,x',y}(v^{f}_\lambda) \geq -p(k) \left\|g\right\|_\infty +(1-p(k))\left[
\m{E}^\eta_{\hat{p},q,\omega,x,(y,F)}\left(v^{\eta}_\lambda \right)-\alpha \left\|g\right\|_\infty-\left\|v^f_{\lambda}-v^{\eta}_{\lambda}\right\|_{F_0}\right].  \label{ineq_E}
\end{equation}
Finally,
\begin{align*}
v^f_\lambda(p,q,\omega) & \geq  - \lambda p(k) \left\|g\right\|_{\infty}+\lambda (1-p(k))g^\eta(\hat{p},q,\omega,x,y)-(1-\lambda) p(k) \left\|g\right\|_{\infty} & ; \text{\eqref{ineq_f}, \eqref{ineq_g} and \eqref{ineq_E}}
\\
&+(1-\lambda)(1-p(k))  \left[\m{E}^\eta_{\hat{p},q,\omega,x,(y,F)}\left(v^{\eta}_\lambda \right) -\alpha \left\|g\right\|_\infty-\left\|v^f_{\lambda}-v^{\eta}_{\lambda}\right\|_{F_0}\right]
\\
&\geq  (1-p(k)) v^\eta_{\lambda}(\hat{p},q,\omega) -p(k) \left\|g\right\|_{\infty} -\alpha \left\|g\right\|_\infty-\left\|v^f_{\lambda}-v^{\eta}_{\lambda}\right\|_{F_0}& ; \eqref{shapley_prime} 
\\
&\geq v^\eta_{\lambda}(\hat{p},q,\omega)-(2\varepsilon+\alpha) \left\|g\right\|_{\infty}-\left\|v^f_{\lambda}-v^{\eta}_{\lambda}\right\|_{F_0} & ; p(k) \leq \varepsilon, \left\|v^\eta_\lambda\right\|_\infty \leq \left\|g\right\|_\infty
\\
&\geq v^\eta_{\lambda}(p,q,\omega)-(4\varepsilon+\alpha) \left\|g\right\|_{\infty}-\left\|v^f_{\lambda}-v^{\eta}_{\lambda}\right\|_{F_0} &; \left\|p-\hat{p}\right\|_1 \leq 2\varepsilon 
\end{align*}
and the result follows. 
\end{proof}

\section{Proof of Theorem \ref{main_thm_2}} \label{sec:uniform}
Let $\Gamma=(K,I,J,\Omega,\rho,g)$ be an absorbing game with incomplete information on one side (since $L$ is a singleton, we omit it in the description of the game). Similarly to Theorem \ref{main_thm}, Assumption \ref{asslip} is supposed to hold without loss of generality. First, note that in the game $\Gamma^f$ constructed in Section \ref{subsec:gamma_f}, Player 1 can guarantee uniformly the limit value. This is indeed a consequence of the existence of the \textit{uniform value} \cite{BGV13}. The idea is to copy strategies of Player 1 in $\Gamma^f$ into strategies in $\Gamma$. An obstacle is that histories in $\Gamma^f$ contain mixed actions of Player 2, while histories in $\Gamma$ contains \textit{pure} actions. Hence, $\Gamma^f$ needs to be modified accordingly. 

Formally, let $C>0$ be given by Proposition \ref{prop:ex_tr}. Let $\varepsilon>0$. Let $\alpha$ be given by Proposition \ref{comparison_e_f}, and $\alpha_0:=\min(\alpha,\varepsilon |K|^{-1/2})$. By Proposition \ref{prop:ex_tr},  there exists an $(\alpha_0,C)$-triangulation $P$ of $\Delta(K)$. By 
Proposition \ref{comparison_e_f}, for all $\lambda \in (0,1]$, for all $(p,\omega) \in \Omega^f=P \times \Omega$,
\begin{equation} \label{eq:approx}
v^f_\lambda(p,\omega)-\varepsilon \leq v^e_\lambda(p,\omega) \leq v^f_\lambda(p,\omega)+\varepsilon.
\end{equation}
The game $\Gamma^{\varphi}$ is described by a state space $P \times \Omega$, action set $X$ for Player 1, action set $J$ for Player 2, transition function $\rho^\varphi$ defined for all $(p,\omega,x,j) \in \Delta(K) \times \Omega \times \Delta(I)^K \times J$ by
\begin{equation*}
\rho^\varphi(p,\omega,x,j):=\left(\sum_{i \in I} \sum_{\omega' \in \Omega}  \rho(\omega'|\omega,i,j) 
\bar{x}^p(i) \sum_{p' \in P} S[p'|p^x(.|i)] \cdot \delta_{(p',\omega')}\right),
\end{equation*}
and payoff function 
\begin{equation*}
g^\varphi(p,\omega,x,j):=\sum_{(k,i) \in K \times I} p(k) x(i|k) g(k,\omega,i,j).
\end{equation*}
The difference between $\Gamma^{\varphi}$ and $\Gamma^f$ is that in $\Gamma^f$, Player 2's action set is $\Delta(J)$. Given a state $(p,\omega) \in P \times \Omega$ and a pair of strategies $(\sigma,\tau)$ in $\Gamma^{\varphi}$, we will denote by $\m{P}^{\varphi}_{p,\omega,\sigma,\tau}$ the probability measure induced by these strategies on the set of infinite histories $(P \times \Omega \times X \times J)^{\m{N}}$. 

It is well-known that in stochastic games with perfect observation of the state, changing the information structure on actions does not modify discounted values. Therefore, for all $\lambda \in (0,1]$, the discounted value of $\Gamma^\varphi$ is equal to $v^f_\lambda$. Hence, $\Gamma^\varphi$ and $\Gamma^f$ have the same limit value, called $w^*$. By \eqref{eq:approx}, the limit value $v^*$ of $\Gamma^e$, which is the same as the limit value of $\Gamma$, satisfies that for all $(p,\omega) \in P \times \Omega$, $w^*(p,\omega) \geq v^*(p,\omega)-\varepsilon$. 
\\
Let $(p,\omega) \in P \times \Omega$. By existence of the uniform value \cite{BGV13}, Player 1 can guarantee uniformly $w^*(p,\omega)$ in $\Gamma^\varphi$. 
Let $\sigma^\varphi$ be a (behavior) strategy that guarantees uniformly $w^*(p,\omega)-\varepsilon \geq v^*(p,\omega)-2\varepsilon$ in $\Gamma^\varphi$. Let $\tau$ be a strategy of Player 2 in $\Gamma$. It is rather easy to ``copy'' $\sigma^\varphi$ into a strategy in $\Gamma$: indeed, at the end of each stage in $\Gamma$, Player 1 can do the triangulation splitting step ``in her head", and play accordingly. Such an argument is made formal by a coupling similar to the one of Section \ref{sec:coupling}. 

We build recursively a process $(H_m,H^\varphi_m)$ on 
\\
$ \cup_{m \geq 1} [K \times (\Omega \times I \times J)^{m-1} \times \Omega] \times \cup_{m \geq 1} [(\Omega \times \Delta(K) \times X \times J)^{m-1} \times \Omega \times \Delta(K)]$ (finite histories in $\Gamma$, finite histories in $\Gamma^\varphi$) in the following way.
\begin{itemize}
\item
For $m=1$, $k$ is drawn from $p$, and $H_1:=(k,\omega)$ and $H^{\varphi}_1:=(p,\omega)$.
\item
For $ m \geq 2$, we define the following variables. 
\begin{itemize} 
\item
$X_m \in X$ is the (realized) mixed action prescribed by strategy $\sigma^{\varphi}$ in $\Gamma^{\varphi}$: conditional to $(H_m,H^{\varphi}_m)$,
$X_m$ has law $\sigma^{\varphi}(H^{\varphi}_m)$,
\item
$J_m \in J$ is the (realized) pure action of Player 2 in $\Gamma$. The law of $J_m$ conditional to $(H_m,H^{\varphi}_m)$ is $\tau(H_m)$,
\item
$I_m$ is the random variable representing Player 1's action, and has law $\overline{X_m}^{P_m}$: 
\begin{equation*}
\forall i \in I \quad \m{P}(I_m=i|H_m,H^{\varphi}_m,X_m):=\overline{X_m}^{P_m}(i),
\end{equation*}
\item
$\Omega_{m+1}$ is the state at stage $m+1$, and is drawn from $\rho(\Omega_m,I_m,J_m)$:
\begin{equation*}
\forall \omega \in \Omega \quad \m{P}(\Omega_{m+1}=\omega|H_m,H^{\varphi}_m,I_m,J_m):=\rho(\omega|\Omega_m,I_m,J_m)
\end{equation*} 
\item
$P_{m+1}$ is drawn from the splitting of the posterior belief ${P_m}^{X_m}(.|I_m)$: 
\begin{equation*}
\m{P}(P_{m+1}=p|H_m,H^{\varphi}_m,I_m)=S[p|{P_m}^{X_m}(.|I_m)]
\end{equation*}       
\item
The updated histories are $H_{m+1}:=(H_m,I_m,J_m,\Omega_{m+1})$ and $H^\varphi_{m+1}:=(H^\varphi_m,X_m,J_m,P_{m+1},\Omega_{m+1})$. 
\end{itemize}
\end{itemize}

This process defines implicitly a strategy of Player 1 in $\Gamma$ and a strategy of Player 2 in $\Gamma^{\varphi}$. Indeed, set
\begin{equation*}
\sigma(H_m)(i):=\m{P}(I_m=i|H_m)
\end{equation*}
\begin{equation*}
\tau^\varphi(H^\varphi_m)(j):=\m{P}(J_m=j|H^\varphi_m).
\end{equation*}
By definition, the law of $(H_m)_{m \geq 1}$ is $\m{P}_{p,\omega,\sigma,\tau}$, and the law of $(H^\varphi_m)_{m \geq 1}$ is $\m{P}^{\varphi}_{p,\omega,\sigma^\varphi,\tau^\varphi}$.
Since $\sigma^\varphi$ guarantees uniformly $v^*(p,\omega)-2\varepsilon$ in $\Gamma^\varphi(p,\omega)$, we have that $\sigma$ guarantees uniformly $v^*(p,\omega)-2\varepsilon$ in $\Gamma(p,\omega)$. 

Assume $p \notin P$. By definition of $\alpha_0$, there exists $p' \in P$ such that $\left\|p-p'\right\|_2 \leq \varepsilon$. Moreover, by the previous construction, there exists $\sigma$ a $2\varepsilon$-uniform optimal strategy in $\Gamma(p',\omega)$. Because $n$-stage payoffs are $\left\|g\right\|_\infty$-Lipschitz, $\sigma$ guarantees uniformly $v^*(p',\omega)-2\varepsilon-\varepsilon \left\|g\right\|_\infty$ in $\Gamma(p,\omega)$. Because $v^*$ is $\left\|g\right\|_\infty$-Lipschitz, 
$\sigma$ is a $(2\varepsilon+2\varepsilon \left\|g\right\|_\infty)$-uniform optimal strategy in $\Gamma(p,\omega)$, and the theorem is proved.  
\section{Perspectives} \label{sec:perspectives}
Most of the tools introduced in this paper extend far beyond the absorbing game model, and offer promising perspectives in various other frameworks. Moreover, the proof inspires the following general approach to tackle the new conjecture stated in subsection \ref{subsec:mertens}:
\begin{itemize}
\item
Given a stochastic game with signals satisfying the assumptions of the new conjecture, consider the auxiliary stochastic game with observed state, where the state space corresponds to the universal belief space \cite{MSZ},
\item
Discretize the belief space into a finite set (e.g. using the triangulation technique of this paper), to obtain another stochastic game with finite state space,
\item
Use \cite{BGV13} to obtain existence of limit value and uniform value in the latter game,
\item
Prove that values in the two stochastic games are close. 
\end{itemize}
Addressing the new conjecture in its full generality is probably too ambitious at first, hence we suggest to investigate the following models as a benchmark. 
\subsection*{Stochastic games with incomplete information on both sides} 
Recall that in the proof of Theorems \ref{main_thm} and \ref{main_thm_2}, the only place where we needed the absorbing assumption is Proposition \ref{opt_concise}. Hence, proving Proposition \ref{opt_concise} without the absorbing absorption would immediately extend Theorems \ref{main_thm} and \ref{main_thm_2} to stochastic games with incomplete information on both sides. A natural starting point would be to consider \textit{recursive games with incomplete information on both sides}, in which payoff in non-absorbing states is always 0. Note that in recursive games with incomplete information on one side, Mertens conjectures were proven true in \cite{RV00}, and this result was extended to a more general signalling structure in \cite{LV14}. 
\subsection*{Repeated games with signals}
In a repeated game with signals, the state never moves. Players do not know the state, and receive private signals at each stage. Such a model satisfies the assumptions of the new conjecture. To prove the latter, the following broad question can be considered as a first step: can the universal belief space be ``triangulated'', in the same spirit as what is done in this paper?
\\

To conclude, two other related problems are stated.
\subsection*{Uniform maxmin} \label{subsec:uniform_maxmin}
In literature, the definition of uniform maxmin often requires in addition that Player 2 should be able to \textit{defend uniformly} the maxmin, that is:
\begin{eqnarray*}
&\forall\varepsilon>0,& \forall \sigma \in \Sigma, \exists \tau \in T, \ \exists n_0 \geq 1, \forall n \geq n_0, 
\gamma_n(p,q,\omega,\sigma,\tau) \leq \maxmin(p,q,\omega)+\varepsilon 
\\
&\text{and}& \quad \gamma_\lambda(p,q,\omega,\sigma,\tau) \leq \maxmin(p,q,\omega)+\varepsilon.
\end{eqnarray*}

Such a property has been proven true in recursive games with incomplete information on one side \cite{RV00}, in stochastic games with imperfect monitoring \cite{RSV03}, and in particular classes of absorbing games with incomplete information on one side \cite{S841,S852,Li20}. The approach of this paper is a good candidate to prove the property in absorbing games with incomplete information on one side. The main difficulty is to adapt Proposition \ref{prop:opt_concise} and prove that if Player 1 can guarantee uniformly some quantity, then she can guarantee it using an $\varepsilon$-concise strategy, up to some error term that vanishes as a function of $\varepsilon$. Once this is done, the coupling argument of Section \ref{sec:coupling} can be adapted to obtain the property.
\subsection*{Complexity results}
The approach of this paper could be of great help to study the computability of the limit value in absorbing games with incomplete information on both sides. Indeed, the size of the triangulation that is needed to obtain approximation of discounted values in $\Gamma$ by discounted values in $\Gamma^f$ can be explicitly bounded. Hence, computability results about limit value in $\Gamma^f$ should immediately produce analogous results for $\Gamma$. The main difficulty is that $\Gamma^f$ is a stochastic game with compact action sets, but as we have seen in Section \ref{subsec:gamma_f}, it is also very regular. In particular, an interesting direction is to extend the technique of \cite{AOB18,OB18} to such games. 
\\
Another direction is to consider Partially Observable Markov Decision Processes (POMDPs), which correspond to 1-Player stochastic games where the player does not know the state, but receives a signal at every stage. Limit value is known to exist by \cite{RSV02}. One can define the game $\Gamma^f$ in a similar way as in this paper, and asks whether the discounted values in $\Gamma^f$ and in the original POMDP are close to each other. This could help understanding the computability of limit value in POMDPs, which remains largely uncharted.
\section*{Acknowledgments}
The author is greatly indebted to Françoise Forges and Sylvain Sorin for their careful rereading and insightful comments. 
 \bibliography{bibliogen.bib}
\end{document}